\documentclass[reqno]{amsart}

\usepackage{amsthm}
\usepackage{amsmath}
\usepackage{amsfonts}
\usepackage{amssymb}
\usepackage{mathrsfs}
\usepackage{graphicx}
\usepackage{amsrefs}
\usepackage[shortlabels]{enumitem}
\usepackage[low-sup]{subdepth} 

\usepackage{color}

\definecolor{light-gray}{gray}{0.4}
\definecolor{very-light-gray}{gray}{0.1}
\newcommand{\gggrayedout}[1]{\mepar{\color{blue}has removed stuff}{
}}

\renewcommand{\emptyset}{\varnothing}
\renewcommand{\epsilon}{\varepsilon}

\def\stirlingtwo#1#2{S^{(2)}(#1,#2)}
\def\stirlingone#1#2{S^{(1)}(#1,#2)}

\def\sequencetwo{{S}^{(2)}}
\def\sequenceone{{S}^{(1)}}
\newcommand\failure{{\rm{Fail}}}
\newcommand\denom{{\rm{Denom}}}
\def\({\left(}
\def\){\right)}

\def\cprime{$'$}
\def\littleo{\operatorname{o}} 
\def\orbitset{{\mathscr{O}}}
\def\fixset{{\mathscr{F}}}
\def\orbit{{\mathsf{O}}}
\def\fix{{\mathsf{F}}}
\def\least{{\mathsf{L}}}

\def\eper{\mathsf{reg}}
\def\lefschetz{L} 
\def\matrices{{\rm{Mat}}}
\def\K{\mathbb K}
\def\Z{\mathbb Z}
\def\Q{\mathbb Q}
\def\R{\mathbb R}
\def\S{\mathbb S}
\def\C{\mathbb C}
\def\N{\mathbb N}
\def\T{\mathbb T}
\def\imag{{\rm{i}}}
\def\eul{{\rm{e}}}
\renewcommand\ge{\geqslant}
\renewcommand\le{\leqslant}
\renewcommand\geq{\geqslant}
\renewcommand\leq{\leqslant}
\renewcommand\subset{\subseteq}
\renewcommand\supset{\supseteq}
\DeclareMathOperator{\lcm}{lcm}
\DeclareMathOperator{\CI}{I}   
\def\sequences{{\mathscr{S}}}
\def\doldB{{\mathsf{B}}}
\def\doldC{{\mathsf{C}}}
\newcommand{\ind}{\operatorname{ind}}
\newcommand{\Per}{\mathscr{Per}}

\newtheorem{theorem}{Theorem}[section]
\newtheorem{lemma}[theorem]{Lemma}
\newtheorem{proposition}[theorem]{Proposition}
\newtheorem{corollary}[theorem]{Corollary}

\theoremstyle{definition}
\newtheorem{definition}[theorem]{Definition}
\newtheorem{example}[theorem]{Example}

\theoremstyle{remark}
\newtheorem{problem}[theorem]{Open problem}
\newtheorem{remark}[theorem]{Remark}

\def\divides{{\mathchoice{\mathrel{\bigm|}}{\mathrel{\bigm|}}{\mathrel{|}}%
       {\mathrel{|}}}}
\def\notdivides{\mathrel{\kern-3pt\not\!\kern4.3pt\bigm|}}
\def\smallnotdivides{\mathrel{\kern-2pt\not\!\kern3.5pt|}}
\def\Divides{\divides\!\divides}
\newcommand\smalldivides{\mathrel{\kern-2pt\kern3.5pt|}}

\def\ftrace{\operatorname{Tr}}   

\DeclareMathOperator{\trace}{Tr}

\newcommand{\dirichlet}{\mathsf d}

\begin{document}
\title{Dold sequences, periodic points, and dynamics}
\author{Jakub Byszewski}
\address{Faculty of Mathematics and Computer Science, Institute of
Mathematics,
Jagiellonian University, ul.
prof. Stanis{\l}awa {\L}ojasiewicza 6, Krak{\'o}w, 30-348, Poland.
Orcid: 0000-0002-2338-5076}
\email{jakub.byszewski@gmail.com}
\author{Grzegorz Graff}
\address{Faculty of Applied Physics and Mathematics,
Gda{\'n}sk University of Technology,
Narutowicza 11/12 Str, Gda{\'n}sk,
80-233, Poland.
Orcid: 0000-0001-5670-5729}
\email{grzegorz.graff@pg.edu.pl}
\author{Thomas Ward}
\address{
School of Mathematics, Statistics and Physics, Newcastle University,
Newcastle upon Tyne, NE1 7RU, U.K.
Orcid: 0000-0002-8253-5767}
\email{tom.ward@newcastle.ac.uk}
\thanks{The first author was supported by the National Science Centre,
Poland under grant no.\ 2016/23/D/ST1/01124.
The second author was supported by the National Science Centre, Poland within the grant Sheng 1 UMO-2018/30/Q/ST1/00228.
The third author was
asked to give the `Wandering Seminar'
at Gda{\'n}sk University of Technology
in February~$2020$, on the topic of
`Periodic points in dynamics'.
The first and second authors had earlier
been preparing a survey on Dold sequences,
and---reflecting the wider history of
these ideas---we discovered that we were
in several cases talking about closely
related
things from different points of view. 
}

 \thanks{}

\date{\today}

\begin{abstract}
In this survey we describe how the so-called Dold congruence
arises in topology, and how it relates to periodic point
counting in dynamical systems.
\end{abstract}

\maketitle
\section{Introduction}
\label{sectionintroduction}
The arithmetic properties of integer sequences
have pervasive connections to questions in number
theory, topology, geometry, combinatorics, dynamical
systems, and doubtless in many other places as well.
These notes are a survey of a chain of ideas whose
origin arguably lies in Fermat's little theorem,
that~$a^p$ is congruent to~$a$ modulo~$p$ for any
integer~$a$ and prime~$p$.
In one direction, it is natural to ask if there
is a canonical way in which this fundamental
congruence is the prime case of a more general
type of
statement about integer sequences.
In another, it is natural to ask of any congruence
if there is a counting argument that
exhibits (in this case)~$a^p-a$ as the
cardinality of a set with a natural~$p$-fold symmetry.
For example, Petersen~\cite{petersen} gave
a proof of Fermat's
little theorem in~$1872$ along exactly these lines, by
counting configurations of~$a$ colours
in~$p$ boxes arranged in a circle
(as a remark attached to a related argument
for Wilson's theorem). Arguments of this sort
have been, and continue to be, repeatedly reinvented or rediscovered---we
refer to Dickson~\cite[p.~75--86]{MR0245499} for some of the early history.
We will discuss questions
that flow from these elementary considerations,
in some cases starting from a quite different question
in topology or dynamics.

A feature of this area is that many of the results
not only have multiple independent
(in some cases, independently repeated) proofs
with different motivations,
but have multiple
histories. One consequence is that the same notion
(or class of integer sequences) may have different names,
and one of the things we try to do in these notes
is to indicate some of these names and to be---at
least internally---consistent in terminology.
In order to explain some of the motivation behind
the different threads here, we to some
extent repeat some of these repetitions,
favouring explanation of motivation over a strictly
logical linear development.

{Another feature of fixed point theory is that
the contexts and applications are widely varied.
The preface to the handbook edited by
Brown {\it{et al.}}~\cite{MR2170491}
talks about `the varied, and not easily classified,
nature of the mathematics that makes up topological fixed point theory'.
Given the existing literature on fixed point theory,
which sprawls in volume (MathSciNet indexes
more than ten thousand items under the primary
subject classification 47H10, `Fixed-point theorems',
for example) and contains several
weighty surveys and handbooks, this modest survey
perhaps needs some justification. Our defence is merely
that the circle of ideas discussed here has some
particularly elementary entry points, and brings
together facets of dynamical systems and topology
in ways that seem interesting---to the authors at
any rate.
}

One of the sources of integer sequences is
counting periodic points in dynamics,
which we will discuss from a
particularly elementary point of view. There are
many ways in which periodic points might
be studied from a more sophisticated point of
view, some of which are briefly outlined later.

Roughly speaking, in studying dynamical systems
one often starts with a map---perhaps
a smooth map on a differentiable manifold---and
uses attributes of the map like hyperbolicity or
local product structure to deduce specification or
closing properties strong enough to construct
periodic points for dynamical reasons (as opposed
to the global topological production of periodic
points mentioned at the end of
Section~\ref{sectionDoldSequences}).
We will describe some of the
consequences of starting at the other end,
motivated by the following
trivial observation (see Remark~\ref{remarkEnglandExampleStory}
for a specific example in the same spirit that originally triggered
the interest of the third-named author in these
kind of questions).
Assume that a map has a single fixed
point. Then, if the number of points fixed by the
second iterate of the map is finite, it must be
a non-negative \emph{odd} number.
This remark sets off a cascade of natural questions
about the combinatorial and analytical properties
of the sequences of periodic point counts of
dynamical systems viewed---initially---simply as
permutations of countable sets.
We explore some of these questions here, with an
emphasis on introducing a broad range of concepts
using fairly simple examples and instances of results,
rather than aiming for maximum generality.

Our focus will be on two closely related---but
motivationally rather different---classes of
integer sequences. Roughly speaking, the first
class of sequences finds its natural home
in topology, the second in dynamical systems
and combinatorics. There is no particular reason
to define one before the other: the class
we describe first is fundamental for a journey
into arithmetic
starting with Fermat's little theorem, the class
we describe second is fundamental for a journey
into dynamics
starting at the same place.

A sequence will be denoted~$a=(a_n)=(a_n)_{n\in\N}$,
and in particular a sequence is always indexed by the
natural numbers~$\N=\{1,2,3,\dots\}$ unless explicitly
indicated otherwise.
With very few exceptions, the maps we are concerned
with are self-maps, so we use phrases like
`a map of~$X$' or `a homeomorphism of~$X$' to
mean a self-map of~$X$, a self-homeomorphism of~$X$,
and so on.\label{pageselfmap}

\section{Dold Sequences}
\label{sectionDoldSequences}

The sequences which we will call
\emph{Dold sequences} as a result of the work
of Dold~\cite{MR724012} have (for example) also been called \emph{sequences having divisibility} in the thesis of Moss~\cite{pm},
\emph{pre-realizable sequences} by Arias de Reyna~\cite{MR2163516},
\emph{relatively realizable sequences} by Neum\"{a}rker~\cite{MR2511223},
\emph{Gauss sequences} by Minton~\cite{MR3195758}, and
\emph{generalized Fermat sequences} or
\emph{Fermat sequences}
by Du, Huang, and Li~\cites{MR1950443, MR2113176}.
Doubtless there are other names,
reflecting the long history and multiple settings
in which they appear.

\begin{definition}[Dold sequence]\label{definitionDoldSequence}\index{Dold sequence}\index{sequence!Dold}
An integer sequence~$a=(a_n)$ is called a
\emph{Dold sequence} if
\begin{equation}\label{equationDoldSequenceMobiusCondition}
\sum_{d\smalldivides n}\mu\(\frac{n}{d}\)a_d
\equiv 0
\end{equation}
modulo~$n$ for all~$n\ge1$.
\end{definition}

Here~$\mu$ denotes the classical
M{\"o}bius\index{Mobius@M\"obius!function}
function, defined by\label{pageMobiusFunction}
\begin{eqnarray*}
\mu(n)
=\left\{
\begin{array}{cl}
1&\mbox{ if }
n=1,\\
0&\mbox{ if~$n$ has a squared factor, and}\\
(-1)^r&\mbox{ if~$n$ is a product of~$r$ distinct primes.}
\end{array}
\right.
\end{eqnarray*}
In particular, if~$n=p$ is a prime,
then~\eqref{equationDoldSequenceMobiusCondition}
is the statement that~$a_p\equiv a_1$ modulo~$p$.
Thus Fermat's little theorem says that the
sequence~$(a^n)$ satisfies the Dold
condition at every prime (and it is
easy to show that it is in fact a Dold
sequence, and in this case the
congruence~\eqref{equationDoldSequenceMobiusCondition}
is normally attributed to Gauss).

We start by making some remarks about
Dold sequences.

\begin{enumerate}[(a)]
\item \label{exDoldref1} If~$A\in\matrices_{m,m}(\Z)$ is an integer matrix,
then the sequence~$\(\trace A^n\)$ is a Dold sequence,
generalizing Fermat's little theorem.
This observation in some form seems to have been known to Gauss, and has been
rediscovered by many others including Browder~\cite{MR0433267},
Peitgen~\cite{MR391074},
and Arnold~\cites{MR2061787,MR2108521,MR2261060,MR2261067};
we refer to notes of Smyth~\cite{MR843194}, Vinberg~\cite{MR2342588},
and Deligne~\cite{MR2506120} for more on this.
For prime~$n$ this was proved by
Sch\"{o}nemann~\cite{MR1578213} in~$1839$.
For~$m=1$ this has been proved
many times; among these are work of
Kantor~\cite{zbMATH02708691},
Weyr~\cite{zbMATH02705034},
Lucas~\cite{zbMATH02703462},
Pellet~\cite{zbMATH02703463},
Thue~\cite{zbMATH02629186},
Szele~\cite{MR28329},
and doubtless many others.
We refer to a note of Steinlein~\cite{MR3654835} for an account and
some of the history, and will discuss this again from
a dynamical point of view in Section~\ref{sectionDynamicalLinearRecurrenceSequences}.
\item The sequence of the number of
fixed points of iterates
of a map is a Dold sequence,  but not
all Dold sequences arise in this way (we will say more about
this later).
\item The original context considered by Dold~\cite{MR724012}
was to show that the sequence of \emph{fixed point indices} of iterations in
 a topological setting
satisfies~\eqref{equationDoldSequenceMobiusCondition}.
\item Marzantowicz and Przygodzki~\cite{MR1696325}
developed a theory of periodic expansions
for integral arithmetic functions, giving a different
characterization of the sequence of fixed point indices and  Lefschetz numbers of iterations of a map.
They also gave probably the
first complete proof of the fact that
sequences like~$\(\trace A^n\)$ are Dold sequences.
\item Many
sequences of combinatorial or
arithmetic origin are Dold sequences,
and in some cases it would be desirable
to have a combinatorial or dynamical
explanation. In simple cases this
is clear, but (for example) it
is not clear why the Bernoulli numerator,
the Bernoulli denominator, and the Euler
sequence~$(\tau_n)$,~$(\beta_n)$,\label{pageBernoullinumeratordenominator}
and~$((-1)^nE_{2n})_{n\ge 1}$ respectively, are Dold
sequences (here~$\frac{\tau_n}{\beta_n}=\vert\frac{B_{2n}}{2n}\vert$
in lowest terms for all~$n\ge1$,\label{pageEulerNumbers} and~$\sum_{n\ge0}E_n\frac{t^n}{n!}=\frac{2}{\eul^{t}+\eul^{-t}}$).
\item If~$(b_n)$ is any integer sequence, then it follows from basic properties of the M\"obius function (see Lemma \ref{CONVID}\ref{CONVID4}) that~$\(\sum_{d\smalldivides n}db_d\)$ is a Dold sequence. Conversely, we prove in Lemma \ref{lemmaConstructingDoldSequences} that all Dold sequences can be obtained in this manner.
\item Many multiplicative sequences are
Dold sequences. One way to construct such sequences is to apply the above method and write such sequences in the form~$(a_n)=\(\sum_{d\smalldivides n}db_d\)$. A simple argument
(see Hardy and Wright~\cite[Th.~265]{MR568909})
shows that~$(a_n)$ is multiplicative if and only if~$(b_n)$ is multiplicative. Choosing $(b_n)$ to be the sequence $b_n=n^k$ of $k$th powers, $k\in\N_0$, shows that the sequence $a_n=\sigma_{k+1}(n)$ of sums of $(k+1)$st powers of divisors is a Dold sequence. On the other hand, the sequence $\sigma_0(n)$ of the number of divisors is not a Dold sequence.
\item Beukers {\it{et al.}}~\cite{MR3855372} characterized
the Dold congruence for the coefficients of the Laurent series
associated to a multi-variable rational function.
\item Samol and van Straten~\cite[Sec.~4]{MR3418804}
considered related congruences of arithmetic interest
for sequences generated as the constant term of
powers of Laurent polynomial; these results were later extended by Mellit
and~Vlasenko \cite{MR3461433}.
\item Kenison {\it{et al.}}~\cite{kenison2020positivity}
have studied positivity questions
for certain holonomic sequences (sequences satisfying
a linear recurrence relation with polynomial
coefficients), relating positivity to questions
on vanishing of periods.
\item While our emphasis is different,
we also mention work of Marzantowicz
and W\'{o}jcik~\cite{MR2302587} addressing related
questions for periodic solutions of certain
ordinary differential equations.
\end{enumerate}

{One of the questions we will be interested
in is the characterization of the intersection of classes of sequences
defined in different ways. The next observation
is a simple instance of this.

\begin{lemma}[Puri {\it{et al.}}~\protect{\cite[Lem.~2.4]{MR1873399}}]\label{lemmanopolynomialisDold}
If a completely multiplicative sequence is a Dold sequence,
then it is the constant sequence with every term equal to~$1$.
If a polynomial sequence is a Dold sequence, then it is a constant.
\end{lemma}

\begin{proof}
If~$(a_n)$ is completely multiplicative, so~$a_{mn}=a_ma_n$
for all~$m,n\ge1$, and satisfies~\eqref{equationDoldSequenceMobiusCondition}
then~$
p^r\divides a_{p^{r-1}}(a_p-1)=a_p^{r-1}(a_p-1)
$
for any prime~$p$ and all~$r\ge1$.
Thus we can write~$a_p=1+pk_p$ for some~$k_p\in\N_0$,
and deduce that~$p^r\divides(1+pk_p)^{r-1}pk_p$
for all~$r\ge1$. This gives~$k_p\equiv0$ modulo~$p^r$
for all~$r\ge1$, so~$a_n=1$ for all~$n\ge1$.

For the second assertion,
assume that~$h(n)=c_0+c_1n+\cdots+c_kn^k$
with~$c_k\neq0$ and~$k\ge1$ is a
polynomial taking integer values on the integers
with the property that~$(h(n))$
is a Dold sequence. After clearing
fractions in~\eqref{equationDoldSequenceMobiusCondition},
we may assume without loss of generality
that the coefficients~$c_0,\dots,c_k$ are integers.
For any prime~$p$ we
have
\[
\tfrac{1}{p^2}\(h(p^2)-h(p)\)\in\(-\tfrac{c_1}{p}+\Z\)\cap\Z
\]
by~\eqref{equationDoldSequenceMobiusCondition},
so~$p\divides c_1$. It follows (since~$p$ is any
prime) that~$c_1=0$.
Thus we can write~$h(n)=c_0+n^2(c_2+c_3n+\cdots+c_kn^{k-2})$
for all~$n\ge1$.
Now let~$p$ and~$q$ be different primes,
so
\[
h({p^2q})-h({pq})-h({p^2})+h(p)
\equiv -h(p^2)+h(p)
\]
modulo~$p^2q$.
Since the left-hand side is independent
of~$q$, this shows that
\[
h(p^2)=h(p)
\]
for all primes~$p$.
This contradicts the hypothesis that~$h$
is a non-constant polynomial.
\end{proof}

}

{
The motivation for some of these questions comes
from the Lefschetz fixed point theorem.\index{Lefschetz fixed point theorem}
For a suitable topological space~$X$
(for example, a finite CW complex or,
more generally, a space whose homology groups
are finitely generated and trivial in all sufficiently
high dimensions) and continuous map~$f\colon X\to X$,
the \emph{Lefschetz number}~$\lefschetz(f)$\label{LefschetzIndex}
is defined
by
\[
\lefschetz(f)=\sum_{k\ge0}(-1)^k\trace(f_{*}\vert_{H_k(X,\Q)}),
\]
the alternating sum of the traces of the
linear maps induced by~$f$ on the
singular homology groups~$H_k(X,\Q)$ of~$X$
with rational coefficients.

\begin{theorem}[Lefschetz fixed point theorem]
If~$f\colon X\to X$ is a continuous map
of a compact CW-complex or, more
generally, a retract of a compact CW-complex,
and~$\lefschetz(f)\neq0$, then~$f$ has a fixed point.
\end{theorem}
}

{More
generally, }Dold~\cite{MR724012}
considered the sequence~$(\ind({f^n},U))$
{of indices}
under the hypothesis that~$f\colon U\to X$ is
a continuous map {on an open subset~$U$}
of an Euclidean neighbourhood
retract~$X$\index{Euclidean neighbourhood retract}
with compact set fixed by each
iterate, and characterised the possible
sequences arising as being those
satisfying~\eqref{equationDoldSequenceMobiusCondition}.
For the case of~$n$ prime, the
direct analogue of Fermat's little theorem,
similar results were shown earlier
by Steinlein~\cite{MR317119}
and by
Zabre\u{\i}ko
and Krasnosel\cprime ski\u{\i}~\cite{MR0315533}.
By using the Leray--Schauder degree to extend
the concepts to infinite dimensions, the Dold
relations have also been shown for suitable
maps on Banach spaces by Steinlein~\cite{MR3392979}.

As pointed out by Smale~\cite[p.~768]{MR228014}
in his influential survey, ``the whole
difficulty of the problem [...] is that
it counts the periodic [points] geometrically,
not algebraically.'' That is, what is natural
to count geometrically in the line of thought
initiated by Lefschetz' theorem is not
the same as simply counting the number
of solutions to the equation~$f(x)=x$.

\subsection{Linear Recurrence Dold Sequences}
\label{sectionDoldLinearRecurrenceSequences}

We will see later that sequences of the form~$\(\trace(A^n)\)$
for~$A\in\matrices_{d,d}(\Z)$ are automatically
Dold sequences for dynamical reasons.
Minton has shown that this is essentially
the only way that a linear recurrence sequence can be a Dold
sequence. Since Minton was working with sequences of rational numbers, he relaxed the conditions in the definition of a Dold sequence, insisting only that the divisibility holds except for powers of finitely many primes. For linear recurrence sequences this is equivalent, however, to considering rational multiples of Dold sequences.

\begin{theorem}[Minton~\protect{\cite[Th.~2.15 \& Rem.~2.16]{MR3195758}}]\label{theoremMinton}
An integer linear recurrence sequence~$(a_n)$ is a rational multiple of
a Dold sequence if and only if it is a \emph{trace
sequence}\index{trace sequence}\index{sequence!trace},
meaning that there is an algebraic number field~$\K$,
rationals~$b_1,\dots,b_r\in\Q$ and
algebraic integers~$\theta_1,\dots,\theta_r\in\K$
with
\[
a_n=\sum_{i=1}^{r}b_i\ftrace_{\K\vert\Q}(\theta_i^n)
\]
for all~$n\ge1$.
\end{theorem}

To see why the `rational multiple' part is important, consider the characteristic sequence~$(a_n)=(0,1,0,1,\dots)$
of even numbers. This is a trace sequence
since
\[
a_n=\tfrac12(1)^n+\tfrac12(-1)^n
\]
for all~$n\ge1$,
but is not a Dold sequence since~$a_2-a_1\not\equiv 0$ modulo~$2$. However,~$(2a_n)$ is a Dold sequence
because, for example, it counts the periodic
points for a map on a set with two elements that
swaps the elements.

We will see a simple instance of
Theorem~\ref{theoremMinton} in
Lemma~\ref{theoremFibonacciCondition} for
a specific historically important
linear recurrence.

\subsection{Generating Functions of Dold Sequences}
\label{sectionGeneratingFunctionsOfDoldSequences}

Write~$\sequences(R)$
for the set of all sequences~$(a_n)$
with values in a subring~$R$ of a
field~$\K$ of characteristic
zero, and~$\sequences$ if the field
plays no role. We define bijections
by
\begin{align*}
\doldB\colon\sequences(\K)&\longrightarrow\sequences(\K)\\
(a_n)&\longmapsto(b_n)
=\Bigl(\frac1n\sum_{d\smalldivides n}\mu\(\frac{n}{d}\)a_d\Bigr)
\intertext{and}
\doldC\colon\sequences(\K)&\longrightarrow\sequences(\K)\\
(a_n)&\longmapsto(c_n)
=
\(\frac1n\bigl(a_n-c_1a_{n-1}-\cdots-c_{n-1}a_1\bigr)\).
\end{align*}
That~$\doldB$ and~$\doldC$ are bijections is clear,
and the inverse maps are given by the
formulas
\begin{align}
a_n&=\sum_{d\smalldivides n}db_d\label{equationAddedInAsInverse}
\intertext{and}
a_n&=c_1a_{n-1}+\cdots+c_{n-1}a_1+nc_n\nonumber
\end{align}
for all~$n\ge1$.
To see that the
formula~\eqref{equationAddedInAsInverse}
does indeed give the inverse of $B$, use Lemma~\ref{CONVID}\ref{CONVID4}.
The sequence~$(c_n)$ is called the
\emph{generating sequence} of~$(a_n)$.
As with much else in these notes, the
various relationships between sequences
described here have multiple names and
appear in many different guises; in settings
close to ours they may be found
in work of Du {\it{et al.}}~\cite{MR2113176} and
Arias de Reyna~\cite{MR2163516}.

\begin{theorem}\label{theoremEulertransforms}
For~$(a_n)\in\sequences$ let~$(b_n)=\doldB((a_n))$
and~$(c_n)=\doldC((a_n))$. Then, in the ring
of formal power series, we have
\begin{equation}\label{equationStartingToIsolateForCoronavirus1}
\exp\bigl(\sum_{n\ge1}\frac{a_n}{n}z^n\bigr)
=
\prod_{n\ge1}(1-z^n)^{-b_n}
=
\bigl(1-\sum_{n\ge1}c_nz^n\bigr)^{-1}.
\end{equation}
\end{theorem}

\begin{proof}
Using the Taylor expansion for the logarithm, we obtain
\begin{align*}
\log\prod_{n\ge1}(1-z^n)^{-b_n}
&=
\sum_{n\ge1}-b_n\log(1-z^n)
=
\sum_{n\ge1}\sum_{k\ge1}b_n\frac{z^{nk}}{k}\\
&=
\sum_{m\ge1}\sum_{n\smalldivides m}nb_n\frac{z^m}{m}
=\sum_{m\ge1}\frac{a_m}{m}z^m,
\end{align*}
giving the first equality.

Write
\[
f(z)=-\exp\bigl(-\sum_{n\ge1}\frac{a_n}{n}z^n\bigr)
\]
for the negative reciprocal of the left-hand expression
in~\eqref{equationStartingToIsolateForCoronavirus1}, and
let~$f(z)=\sum_{n\ge0}c_n'z^n$ be its
Taylor expansion. Comparing
the values at zero
gives~$c_0'=-1$, and the logarithmic
derivative is given by
\[
\frac{f'(z)}{f(z)}=-\sum_{n\ge1}a_nz^{n-1}.
\]
Hence
\[
f'(z)=\sum_{n\ge1}nc_n'z^{n-1}=
\(\sum_{n\ge1}a_nz^{n-1}\)\cdot
\(1-\sum_{n\ge1}c_n'z^n\),
\]
giving the equalities
\[
nc_n'=a_n-c_1'a_{n-1}-\cdots-c_{n-1}'a_1
\]
for all~$n\ge1$. This proves that~$c_n'=c_n$
for all~$n\ge1$, completing the proof.
\end{proof}

In the dynamical context relations
like~\eqref{equationStartingToIsolateForCoronavirus1}
may be used to prove integrality of certain sequences;
we refer to Jaidee {\it{et al.}}~\cite{MR4002553}
for examples.
By M{\"o}bius inversion
we may deduce from the
definition of~$\doldB$ a classical
relation due to M{\"o}bius~\cite{MR1577896},
\[
\eul^x=\prod_{n\ge1}(1-x^n)^{-\mu(n)/n},
\]
valid for~$\vert x\vert<1$.
In a more arithmetic direction, applying this
to the Artin--Hasse exponential series
(defined in~\cite{MR3069494})\index{Artin--Hasse series}
\[
E_p(x)=\exp\(x+\frac{x^p}{p}+\frac{x^{p^2}}{p^2}+\frac{x^{p^3}}{p^3}+\cdots\)
=
\exp\bigl(\sum_{n\ge1}\frac{a_n}{n}x^n\bigr)
\]
gives the corresponding~$(b_n)$
by the formula
\[
b_n=
\begin{cases}
\frac{\mu(n)}{n}&\mbox{if }p\notdivides n;\\
0&\mbox{otherwise}.
\end{cases}
\]
Thus Theorem~\ref{theoremEulertransforms}
gives
\[
E_p(x)=\prod_{p\smallnotdivides n}(1-x^n)^{-\mu(n)/n},
\]
from which
the well-known fact that the coefficients
of the Taylor expansion of~$E_p$
are~$p$-adic integers
may be deduced
by using the~$p$-adic continuity
of the binomial coefficient polynomials,
or from `Dwork's lemma'
(see Koblitz~\cite[Ch.~IV.2]{MR0466081} for the latter
argument).

Theorem~\ref{theoremEulertransforms} gives a
constructive alternative definition of Dold
sequences. This has a parallel in
Section~\ref{sectionRealizableSequences},
where it is expressed as the relation between
counting closed orbits and counting periodic points.

\begin{lemma}\label{lemmaConstructingDoldSequences}
For~$(a_n)\in\sequences$ let~$(b_n)=\doldB((a_n))$
and~$(c_n)=\doldC((a_n))$.
Then the following are equivalent:
\begin{enumerate}[\rm(a)]
\item $(a_n)$ is a Dold sequence;
\item $(b_n)\in\sequences(\Z)$;
\item $(c_n)\in\sequences(\Z)$.
\end{enumerate}
\end{lemma}

\begin{proof}
The sequence~$(a_n)$ is a Dold
sequence if and only if the values of~$(b_n)$
are integers by~\eqref{equationAddedInAsInverse}.
Writing
\[
g(z)=
\prod_{n\ge1}(1-z^n)^{b_n}
=
\bigl(1-\sum_{n\ge1}c_nz^n\bigr),
\]
it is clear that if all the powers~$b_n$
are integers, then so are all the
coefficients~$c_n$. For the reverse direction,
assume that~$c_n\in\Z$ for all~$n\ge1$
and assume that~$b_1,\dots,b_{m-1}$
are integers.
The power series of~$g$
computed modulo~$z^{m+1}$
gives
\[
g(z)
\equiv
\prod_{n=1}^{m}(1-z^n)^{b_n}
\equiv
(1-b_mz^m)\prod_{n=1}^{m-1}(1-z^n)^{b_n}
\equiv
\prod_{n=1}^{m-1}(1-z^n)^{b_n}-b_mz^m
\]
modulo~$z^{m+1}$. Since~$g$
has integer coefficients, we deduce that~$b_m$
is an integer, proving that all~$b_n$ are integers
by induction.
\end{proof}

The different ways of expressing the property of
being a Dold sequence facilitate a way to decompose
a Dold sequence into elementary periodic sequences as
follows.
For~$d\in\N$ we define
\[
\eper_d(n)
=
\begin{cases}
d&\mbox{if }d\divides n;\\
0&\mbox{if }d\notdivides n.
\end{cases}
\]
The following decomposition---a different
way of expressing the transformation~$\doldB$---is
given by the following result of
Jezierski and Marzantowicz~\cite[Prop.~3.2.7]{MR2189944}.

\begin{proposition}\label{propositionDoldPeriodicDecomposition}
Let~$(a_n)\in\sequences$ and let~$(b_n)=\doldB((a_n))$
as above. Then
\begin{equation}\label{perexp}
a_n=\sum_{k=1}^{\infty}b_k\eper_k(n).
\end{equation}
Moreover,~$(a_n)$ is an integer Dold sequence
if and only if~$(b_n)\in\sequences(\Z)$.
\end{proposition}

This is an immediate consequence of the  formula for~$\doldB$ and
Lemma~\ref{lemmaConstructingDoldSequences}, and in the
context of dynamics this may be thought
of as building up a dynamical system as
a union of closed orbits.
The class of integer multiples
of elementary periodic sequences
is closed under pointwise multiplication,
since we have
\begin{equation}\label{equationProductElementaryPeriodicSequences}
\eper_k\cdot\eper_{\ell}=\gcd(k,\ell)\eper_{\lcm(k,\ell)}.
\end{equation}
When we study dynamical systems with
finitely many points of each period,
it will be clear that a bound on the
number of periodic points is equivalent
to a bound on the possible
lengths of closed orbits. In the more
general context of Dold sequences this
is expressed as follows.

\begin{lemma}[Babenko {\&} Bogaty\u{\i}~\cite{MR1130026}]\label{lemmaDoldPeriodicIffBounded}
A Dold sequence is bounded if and only
if it is periodic. If this is the case, it can be written as a finite linear combination with integer coefficients of elementary periodic sequences $\eper_k$.
\end{lemma}

\begin{proof}
Let~$(a_n)_n$ be a Dold sequence
with~$|a_n|<T$ for all~$n\ge1$.
Then
\[
|b_n|=\frac{1}{n}
|\sum_{k\smalldivides n}\mu(k)a_{k/d}|\leq\frac{T}{n}
\sum_{k\smalldivides n}|\mu(k)|
\leq\frac{T}{n}\sigma_0(n),
\]
where~$\sigma_0(n)$ is the
number of positive divisors of~$n$.
It is well-known that for any~$\beta>0$
we have~$\sigma_0(n)=\littleo(n^{\beta})$
(see Hardy and Wright~\cite[Sec.~18.1]{MR568909}),
so~$|b_k|\rightarrow 0$
as~$k\rightarrow \infty$.
Since~$b_n\in\Z$, we deduce
that~$b_n=0$ for all but finitely
many~$n\in\N$, completing the proof.
\end{proof}

The representation of a Dold sequence in the
form~\eqref{perexp} is called
its {\it periodic expansion}, and is useful in
many problems in periodic point theory.
For example, in the periodic expansion of the
Lefschetz numbers of iterations of Morse--Smale
diffeomorphisms, each~$(\eper_k)$ for odd~$k$
represents a periodic orbit of minimal period~$k$.
This observation allows one to easily determine
the so-called
{\it minimal set of Lefschetz  periods}~${\rm MPer}_L(f)$
for Morse--Smale  diffeomorphisms (see Graff {\it{et al.}}~\cite{MR3921437}),  studied by
Guirao, Llibre, Sirvent and other authors  (we refer to~\cite{MR3235353} and the references therein
for more on this).

The language of
periodic expansions turns out to be useful in some
dynamical problems in magnetohydrodynamics. This
applies to the flow in magnetic flux tubes, which
can be studied via
the discrete field line mapping~$\varphi$ through
the consecutive horizontal cross-sections of the tube.
The bounds for coefficients of the periodic expansion of
the fixed point
indices of iterations of~$\varphi$ at periodic points provide
a class of topological constraints sought
by the astrophysicists Yeates {\it{et al.}}~\cite{yeates}
to explain unexpected behaviour in certain
resistive-magnetohydrodynamic simulations
of magnetic relaxation.
We refer to work of Graff {\it{et al.}}~\cite{MR3846955}
for more on this.


\subsection{Other Characterizations}
\label{sectionOtherCharacterizations}

Andr\'{a}s~\cite[Rem.~2]{MR2790978}
noticed that in the definition of
the Dold
property the M\"obius function
can be replaced by the Euler totient function~$\phi$
defined as usual by~$\phi(n)=\vert\{k\mid1\le k\le n; \; \gcd(k,n)=1\}\vert$.
That is, an integer sequence~$(a_n)$ is a Dold
sequence if and only if
\begin{equation}\label{phi-Dold}
\sum_{d\smalldivides n}\phi\(\frac{n}{d}\)a_d
\equiv 0
\end{equation}
modulo~$n$ for all~$n\ge1$.
Recently, W\'ojcik~\cite{wojcik}
proved that in fact the
Dold property can be characterized
with the M\"obius function
replaced by any other integer-valued
function satisfying some natural constraints.

\begin{theorem}[W\'ojcik~\cite{wojcik}]\label{replacement}
Let~$\psi\colon\N\to\Z$ be a
function satisfying the two conditions
\begin{equation}\label{conditions-psi}
\left\{
\begin{aligned}
\psi(1)&=\pm 1;\\
\sum_{k\smalldivides n}\psi(k)&\equiv 0\pmod{n}\mbox{ for all }n\in\N.
\end{aligned}
\right.
\end{equation}
Then~$(a_n)$ is a Dold sequence if and only if
\begin{equation}\label{Doldrelpsi}
\sum_{k|n}\psi\Big(\frac{n}{k}\Big)a_{k}\equiv0
\end{equation}
modulo~$n$ for all~$n\ge1$.
\end{theorem}

Notice that~\eqref{conditions-psi} is satisfied
by the M\"obius function, as the sum vanishes for~$n>1$
and is~$1$ for~$n=1$,
and is satisfied by the Euler function as the sum
in that case is equal to~$n$.

\begin{remark}\label{graffmobiuslattice}
In the recent preprint~\cite{graffgeneralizeddold}
Graff {\it{et al.}}
generalize the notion of Dold sequences
to the setting of partially ordered sets.
In this approach classical Dold sequences are the
special case in which the partial order
is given by the relation of divisibility.
The M\"obius function of a
partially ordered set is an old concept,
and became a central tool in
combinatorics following Rota's seminal work~\cite{MR174487}.
We will mention later in Remark~\ref{richardthing}
how this plays a role in studying orbit growth
for group actions.
\end{remark}

\subsection{Polynomial Dold Sequences}
\label{sectionPolynomialDoldSequences}

The notion of Dold sequence has been generalized
to sequences of polynomials by Gorodetsky~\cite{MR4015520}.
For~$n\in\N$, we define
the polynomial~$[n]_q$\label{pagenqdeffinition}
in the variable~$q$ by
\[
[n]_q=\frac{q^n-1}{q-1}=1+q+\cdots+q^{n-1}\in \mathbb Z[q].
\]
Clearly~$\lim_{q\to1}[n]_q=n$, so
this may be thought of as
the~$q$-analogue of the positive integer~$n$.
Indeed, if~$a>b$ with~$a\equiv b$ modulo~$n$, then
\[
[a]_q=\frac{q^a-1}{q-1}=\frac{q^b(q^{a-b}-1)+q^b-1}{q-1}
\equiv\frac{q^b-1}{q-1}=[b]_q
\]
modulo~$[n]_q$ as a congruence in the polynomial
ring~$\Z[q]$.

\begin{definition}\label{qDold}
A sequence of polynomials~$(a_n(q))$
with terms in~$\Z[q]$ is said to satisfy the
\emph{$q$-Gauss congruences}\index{qDoldcongruences@$q$-Dold congruences}
(or \emph{$q$-Dold congruences} in our terminology)
if
\begin{equation}\label{q-congruences}
\sum_{d\smalldivides n}\mu(d)a_{n/d}(q^d)\equiv 0
\end{equation}
modulo~$[n]_q$ in~$\Z[q]$ for all~$n\in\N$.
\end{definition}

The fact that a given sequence~$(a_n(q))$
satisfies the~$q$-Dold congruences
implies that the integer
sequence~$(a_n(1))$
satisfies the Dold congruences.
As a consequence, some
polynomial techniques are then able
to be applied to
find results about integer Dold sequences.
{In particular, tools like formal
differentiation become available. We refer
to work of Guo and Zudilin~\cite{MR3910798}
and Straub~\cite{MR3896053} for more on this.}

\section{Periodic Points}
\label{sectionRealizableSequences}

\subsection{Relating periodic points to closed orbits}
\label{Relating periodic points to closed orbits}

We will mostly be concerned with a `dynamical system',
which initially means simply a map~$T\colon X\to X$
on a set~$X$.
The two fundamental notions we will deal with are
{periodic points} and {closed
orbits}.\index{dynamical!system}

\begin{definition}\label{definitionClosedOrbits}
Let~$T\colon X\to X$ be a map.
\begin{enumerate}[\rm(a)]
\item The \emph{orbit} of a
point~$x\in X$\index{orbit}
under~$T$
is the set\label{pageOrbitSet}
\[
\orbitset_T(x)=\{T^kx\mid k\in\N_0\}.
\]
If~$T^nx=x$ for some~$n\ge1$ then~$x$
is a \emph{periodic point}\index{periodic point}
and
the set~$\orbitset_T(x)$ is a\index{closed orbit}
\emph{closed orbit} of \emph{length}~$\vert\orbitset_T(x)\vert$.
We write~$\orbit_T(n)$\label{pageOrbitT(n)}
for the
number of closed orbits of length~$n\ge1$.
\item If~$T^kx=x$ for some~$k\ge1$ and~$x\in X$,
then~$x$ is a \emph{periodic point}.
The \emph{least period} (or minimal period)
of a periodic\index{least period}
point~$x$ is~$\min\{k\in\N\mid T^kx=x\}$.
We write\label{pageFixSet}
\[
\fixset_T(n)=\{x\in X\mid T^nx=x\}
\]
for the set of points fixed by the~$n$th
iterate of~$T$,
and\label{pageFixT(n)}
\[
\fix_T(n)=\vert\fixset_T(n)\vert
\]
for the number of points fixed by the~$n$th
iterate of~$T$, or (equivalently) the
number of points that are periodic with
least period dividing~$n$.
We also write~$\least_T(n)$ for the
number of points in~$\fixset_T(n)$
with least period~$n$, so~$n\orbit_T(n)=\least_T(n)$
for all~$n\ge1$.
\end{enumerate}
\end{definition}

We will always
assume that~$\orbit_T(n)<\infty$ for all~$n\ge1$,
so that all the quantities in
Definition~\ref{definitionClosedOrbits} are finite.
When we refer to a `map' or a `dynamical system', we
make this assumption,
and~$(X,T)$\label{pageXTdefined}
always means such a system.

Clearly the sequences~$\(\orbit_T(n)\)_{n\ge1}$
and~$\(\fix_T(n)\)_{n\ge1}$ determine each other.

\subsection{Dirichlet Series and Arithmetic Functions}
\label{Dirichlet Series and Arithmetic Functions}

\begin{definition}\label{definitionDirichletConvolution}
The \emph{Dirichlet convolution} of\index{Dirichlet!convolution}
functions\index{convolution}\index{g}~$f,g\colon\N\to\C$ is the
function~$f{\ast} g\colon\N\to\C$ defined by
\[
({f}{\ast}{g})(n)
=
\sum_{d\divides n}
{f}(d){g}\(\frac{n}{d}\)\!\!.
\]
\end{definition}

Convolution is commutative and associative, and the following
properties of the M{\"o}bius function are easily checked.

\begin{lemma}\label{CONVID}\index{convolution!identity}
Define the function~$\CI\colon\N\to\C$ by~$\CI(1)=1$ and~$\CI(n)=0$
for all~$n>1$.
\begin{enumerate}[\rm(a)]
\item $f{\ast}{\CI}={\CI}{\ast} f=f$
for any function~$f\colon\N\to\C$.
\item If~$f\colon\N\to\C$ has~$f(1)\neq0$,
then there is a unique
function~$g\colon\N\to\C$ such that~${f}{\ast}{g}=\CI$. This function is
denoted~${f}^{-1}$.\index{convolution!inverse}
\item The M{\"o}bius function is multiplicative,\index{multiplicative function}
meaning that~$\mu(mn)=\mu(m)\mu(n)$ if~$\gcd(m,n)=1$.
\item\label{CONVID4}  The M{\"o}bius function satisfies
\[
\sum_{d\divides n}\mu(d)
=
\begin{cases}
1&\mbox{if }n=1;\\
0&\mbox{if }n>1.
\end{cases}
\]
\item If~$u(n)=1$ for all~$n\in\N$, then~$u^{-1}=\mu$.
\item For functions~$f,g\colon\N\to\C$ the following statements are equivalent:
    \begin{enumerate}
    \item[\rm{(i)}] $\displaystyle f(n)=\sum_{d\divides n}g(d) \text{ for all } n\in \N;$
    \item[\rm{(ii)}] $\displaystyle g(n)=\sum_{d\divides n}f(d)\mu\(\frac{n}{d}\)\text{ for all } n\in \N.$
    \end{enumerate}
\end{enumerate}
\end{lemma}

\begin{proof}
The first statement~(a) is clear from the definition of
convolution.

The equation~$({f}{{\ast}}{g})(1)
={ f}(1){
g}(1)$ determines~${ g}(1)$.
Assuming that~${g}(k)$ has been
determined for~$1\le k<n$,
the equation
\[
({f}{\ast}{g})(n)
=
{f}(1){g}(n)
+
\sum_{1<d\divides n}{f}(d){g}\(\frac{n}{d}\)
\]
determines~${g}(n)$ uniquely, showing~(b).

Let~$m$ and~$n$ be
integers with~$\gcd(m,n)=1$,
and factorize~$m$ and~$n$ as products
of prime powers. Clearly no prime can appear
as a factor of both~$m$ and~$n$.
If any prime appears
with exponent~$2$ or more, then
both sides of the equation~$\mu(mn)=\mu(m)\mu(n)$ are
zero. If~$m$~($n$) is a product of~$k$ (resp.~$\ell$) distinct
primes, then~$mn$ is a product of~$k+\ell$ primes,
since~$m$ and~$n$ are co-prime.
Thus~$\mu(m)=(-1)^{k}$ and~$\mu(n)=(-1)^{\ell}$, so
\[
\mu(mn) = (-1)^{k+\ell} = \mu(m)\mu(n),
\]
showing~(c).

The right-hand side of~(d) is multiplicative.
We claim that the left-hand side is also
multiplicative. If~$m$ and~$n$ are
coprime, then
\[
\sum_{d\divides mn}\mu(d)=\sum_{d_1\divides m}\sum_{d_2\divides n}\mu(d_1d_2)
=\sum_{d_1\divides m}\mu(d_1)\sum_{d_2\divides n}\mu(d_2)
\]
since~$\gcd(d_1,d_2)=1$
for~$d_1\divides m$ and~$d_2\divides n$,
and~$\mu$ is multiplicative.
Since both sides of the claimed identity are multiplicative,
it is enough to verify it when~$n=p^r$, where we
have
\[
\sum_{d\divides p^r}\mu(d)=\mu(1)+\mu(p)=0,
\]
showing~(d). The statement in~(e) is
simply the formula in~(d) again.

The claim in~(f) is usually called the
\emph{M{\"o}bius inversion formula}.\index{Mobius@M\"obius!inversion}
We may write
\[
f(n)=\sum_{d\divides n}g(d)
\]
as the convolution identity~$f=g{\ast}u$.
By associativity of convolution and~(e),
we therefore have~$f{\ast}\mu=g{\ast}u{\ast}\mu=
g{\ast}\CI=g$, so
\[
g(n)=\sum_{d\divides n}f(d)\mu\(\frac{n}{d}\).
\]
The converse direction is similar: convolve~$g=f{\ast}\mu$
with~$u$.
\end{proof}

The following relationship between periodic
points and closed orbits is
highlighted in Smale's influential
survey~\cite[Sect.~I.4]{MR228014}.

\begin{lemma}\label{lemmaFundamentalRelationship}
Let~$T\colon X\to X$ be a map. Then
\begin{equation}\label{equationfixintermsoforbits}
\fix_T(n)=\sum_{d\divides n}d\orbit_T(d)
\end{equation}
and so
\begin{equation}\label{equationMobius1}
\orbit_T(n)=\frac1n \sum_{d\divides n}\mu\(\frac{n}{d}\)\fix_T(d).
\end{equation}
\end{lemma}

\begin{proof}
Any point fixed by~$T^n$ must live on a
closed orbit of length~$d$
for some~$d$ dividing~$n$,
so~$\fixset_T(n)$ is the disjoint union of
these closed orbits, and the
number of points that live
on closed orbits of length~$d$ is~$d\orbit_T(d)$,
showing~\eqref{equationfixintermsoforbits}.
The equivalence of~\eqref{equationfixintermsoforbits}
and~\eqref{equationMobius1} follows at
once from Lemma~\ref{CONVID}(f).
\end{proof}

\begin{definition}\index{zeta function!Artin--Mazur, dynamical}
The (Artin--Mazur or dynamical) \emph{zeta
function}~$\zeta_T$\label{pageZetaFunction}
is the formal
power series defined by
\begin{equation}\label{equationDefinesZetaFunction}
\zeta_T(z)=\exp\(\sum_{n\ge1}\frac{z^n}{n}\fix_T(n)\).
\end{equation}
The \emph{dynamical Dirichlet series}~$\dirichlet_T$\index{dynamical Dirichlet series}\index{Dirichlet!dynamical series}\label{pageDynamicalDirichletSeries}
is defined by
\begin{equation}\label{equationDefinesOrbitDirichletSeries}
\dirichlet_T(s)=\sum_{n\ge1}\frac{\orbit_T(n)}{n^s}.
\end{equation}
\end{definition}

Some remarks will help to familiarize these
formal definitions.
\begin{enumerate}[\rm(a)]
\item The basic relation~\eqref{equationfixintermsoforbits}
or, equivalently,~\eqref{equationMobius1},
may be expressed just as in
Theorem~\ref{theoremEulertransforms} in terms of these two generating
functions by
thinking of the
collection of all closed
orbits for a system~$(X,T)$
as a disjoint union of individual orbits,
to give\index{Euler product}
the identity
\begin{equation}\label{equationProductOverOrbits}
\zeta_T(z)=\prod_{n\ge1}\bigl(1-z^n\bigr)^{-\orbit_T(n)}
=\prod_{\tau}\bigl(1-z^{\vert\tau\vert}\bigr)^{-1},
\end{equation}
where the product is taken over all closed orbits~$\tau$ of~$T$, and
the resulting identity\label{pageRiemannZetaFunction}
\begin{equation}\label{andtheycamefromeverywhere}
\dirichlet_T(s){\boldsymbol\zeta}(s+1)=\sum_{n\ge1}\frac{\fix_T(n)}{n^{s+1}},
\end{equation}
where~$\boldsymbol\zeta$ is the classical (Riemann) zeta function.\index{zeta function!Riemann}
\item If there is an exponential bound of the
form
\begin{equation}\label{equationCauchyHadamardRadiusOfConvergence}
\limsup_{n\to\infty}\fix_T(n)^{1/n}=R<\infty
\end{equation}
then
the dynamical zeta function is a convergent
power series with radius of
convergence~$\rho(\zeta_T)=1/R$\label{pageRadiusOfConvergence}
by the Cauchy--Hadamard theorem.\index{Cauchy--Hadamard theorem}
\item Write~$s=\sigma+\imag t$, and
hence~$\vert n^s\vert= n^{\sigma}$.
If a Dirichlet series converges absolutely for~$s_0=\sigma_0+\imag t_0$,
then it converges absolutely
for all $s$ with~$\sigma>\sigma_0$. It follows
that if~$\dirichlet_T(s)$ does converge for some~$s\in\C$
but does not converge for all~$s\in\C$, then there is
a real\label{pageDefineabscissaofconvergence}
number~$\sigma_{\text{abs}}\index{Dirichlet!series!abscissa of absolute convergence}
=\sigma_{\text{abs}}(\dirichlet_T)$ (the `abscissa of
absolute convergence') with the property that~$\dirichlet_T(s)$
converges absolutely if~$\sigma>\sigma_{\text{abs}}$ but
does not converge absolutely if~$\sigma<\sigma_{\text{abs}}$.
Thus we expect~$\dirichlet_T$ to define a
function on some right-half plane in~$\C$
if there is a polynomial bound~$\orbit_T(n)\le An^B$
on the growth in closed orbits.

\item A striking example to illustrate settings in which
the dynamical Dirichlet series may be useful comes
from
the quadratic map~$\alpha:x\mapsto1-cx^2$ on the
interval~$[-1,1]$ at the Feigenbaum
value~$c=1.401155\cdots$.
As pointed out by Ruelle~\cite{MR1920859},
this map has dynamical zeta function
\[
\zeta_{\alpha}(z)=\prod_{n=0}^{\infty}\bigl(1-z^{2^n}\bigr)^{-1}=\prod_{n=0}^{\infty}\bigl(1+z^{2^n}\bigr)^{n+1},
\]
which satisfies the Mahler functional equation~$\zeta(z^2)=(1-z)\zeta(z)$
and hence admits a natural boundary.
In contrast, Everest {\it{et al.}}~\cite{MR2550149}
point out that the dynamical Dirichlet series
is given by~$\dirichlet_{\alpha}(z)=\frac{1}{1-2^{-z}}$,
with readily understood analytic behaviour.

\end{enumerate}

\begin{remark}\label{remarkEnglandExampleStory}
The third author's
interest in the kind of questions discussed
here to some extent began with an obscure error
in a paper of England and Smith~\cite{MR307280}.
In order to show that an automorphism of a
solenoid can have an irrational zeta function,\index{zeta function!irrational}
they construct an explicit example and claim
its zeta function is
\begin{equation}\label{equationEnglandAndSmithExample}
\exp\(4z+4z^2+\sum_{k=3}^{\infty}(7^k-3^k)\frac{z^k}{k}\).
\end{equation}
Many years ago he was looking at the paper
for other reasons, and noticed that this
could not be the zeta function of any
map because~$z\mapsto\frac{1-3z}{1-7z}$ is clearly
a dynamical zeta function
(corresponding to the
sequence of coefficients~$(7^n-3^n)$,
or counting the periodic points of the map
dual to~$x\mapsto\frac{7}{3}x$ on~$\Z[\frac{1}{3}]$),
and the basic Dold relation
shows that it is not really possible to
change finitely many terms of a sequence
while preserving the property of counting periodic points for a map.
Indeed, if~$\zeta_T(z)$ is given by~\eqref{equationEnglandAndSmithExample},
then we would
have~$\(\fix_T(n)\)=\(4,8,316,2320,16564,116920,\dots\)$.
Applying the M{\"o}bius transform~\eqref{equationMobius1}
then gives
\[
\(\orbit_T(n)\)=\(4,2,104,578,3312,\tfrac{58300}{3},\dots\),
\]
which is impossible.
In fact irrational zeta functions are not only
possible but common---indeed, in a certain
sense, generic---for compact group automorphisms.
We refer to work of Everest {\it{et al.}}~\cites{MR2180241, MR1458718, MR1619569}
for more on this. Indeed, it seems that a typical
compact group automorphism will admit
a natural boundary for the zeta function;
see Bell {\it{et al.}}~\cite{MR3217030}
and the survey~\cite{MR3330348}
for more on this.
Bowen and Lanford~\cite{MR0271401}
had earlier pointed out a large class
of symbolic dynamical systems with
irrational zeta functions by constructing
uncountably many symbolic systems with distinct
zeta functions, and showing that there are
only countably many rational dynamical
zeta functions.
\end{remark}

\begin{remark}\label{richardthing}
As mentioned in Remark~\ref{graffmobiuslattice},
the fundamental congruence~\eqref{equationDoldSequenceMobiusCondition}
can be interpreted in any partially ordered set.
Miles and Ward~\cites{MR2465676,MR2650793}
used a similar extension of Lemma~\ref{lemmaFundamentalRelationship}
for closed orbits of group actions and
bounds on the values of the M\"obius function
on the lattice of finite-index subgroups of nilpotent
groups to find asymptotics for orbit growth
properties of some simple algebraic actions of
abelian and nilpotent groups.
\end{remark}

\section{Universal Dynamical Systems}
\label{A Universal Example}

Periodic points in dynamics arise in several
settings:
\begin{itemize}
\item In various classical mechanical settings,
periodic orbits represent possible repetitive motions.
For example, in the three body problem great
efforts have gone into establishing the existence
of infinitely many periodic orbits.
\item More generally, the Arnold conjecture in
symplectic geometry concerns how the additional rigidity
of symplectic manifolds impacts on the number
of closed orbits; we refer to Arnold's
famous book of problems~\cite[1972-33]{MR2078115} for
more on this important problem.
\item As invariants of suitable notions of equivalence.
For example, if~$T\colon X\to X$ and~$S\colon Y\to Y$
are homeomorphisms of compact metric spaces (`topological
dynamics') then a \emph{topological conjugacy}
between the systems~$(X,T)$ and~$(Y,S)$, meaning
a homeomorphism~$\phi\colon X\to Y$
with the property that~$\phi\circ T=S\circ\phi$,
implies that~$\fix_T(n)=\fix_S(n)$ for all~$n\in\N$.
\item As indicators of complexity. For example,
under certain natural conditions the logarithmic
growth rate of~$\fix_T(n)$ is related to
the topological entropy of~$T$.
We refer to the work of Katok~\cite{MR573822}
for foundational results of
this sort.
\item In a different
direction, work of Markus and Meyer~\cite{MR556887}
shows that certain `exotic' orbit growth patterns
that arise in one-dimensional solenoids
are a generic feature of high-dimensional Hamiltonian
dynamical systems.
\item As a product of the presence of other
phenomena via closing lemmas, specification properties,
and so on.
\item As an insight into typical behaviour in
various contexts. For example, Artin and
Mazur~\cite{MR176482}
show there is a~$C^k$-dense set of~$C^k$-maps
on a compact smooth manifold without boundary
with an exponential bound on the growth
in~$\fix(n)$ (a result later extended
and strengthened by Kaloshin~\cite{MR1726706}).
\item In contrast, Kaloshin~\cite{MR1757015}
showed that arbitrarily
fast growth in~$\fix(n)$ is Baire generic
in the space of~$C^2$ or smoother diffeomorphisms.
\item Kaloshin and Hunt~\cite{MR2276768}
showed that for a measure-theoretic
notion of generic (`prevalent') the growth
is typically only a little faster than
exponential, with a growth bound
of the shape~$\exp(Cn^{1+\delta})$.
\item Counting periodic points in algebraic
systems can lead to subtle Diophantine problems
(we refer to work of Lind~\cite{MR684244}
on quasihyperbolic toral automorphisms
and work of Everest {\it{et al.}}~\cite{MR1461206,MR2339472}
for examples; very similar problems occur also for endomorphisms of algebraic groups (and related maps) in positive characteristic~\cite{MR3894433, MR4062561}).
\end{itemize}

A less familiar question is to ask not
for what is \emph{typical} in various contexts
but to ask for what is \emph{possible}.

\subsection{Combinatorial Maps}
\label{sectionCombinatorial}

\begin{definition}\label{definitionRealisable}
A sequence~$(a_n)$ with terms in~$\N_0$ is
called \emph{realizable} if there is some\index{realizable sequence}
map~$T\colon X\to X$ with~$a_n=\fix_T(n)$
for all~$n\ge1$.
\end{definition}

It is clear that~\eqref{equationMobius1}
gives the only constraint on the
periodic points of
a bijection. That is, if~$(o_n)$ is any
sequence with values in~$\N_0$ then there
is a bijection~$T\colon X\to X$ with~$\orbit_T(n)=o_n$
for all~$n\in\N$, simply by defining~$X$
to be the disjoint union of~$o_n$ many
closed orbits of length~$n$
for each~$n\in\N$.
Thus a sequence~$(a_n)$ is
realizable if and only if
\begin{equation}\label{equationMainCondition}
\frac1n\sum_{d\divides n}\mu\(\frac{n}{d}\)a_d\in\N_0
\end{equation}
for all~$n\ge1$. Notice that this requires
two conditions:
\begin{itemize}
\item (Non-negative) The sum over divisors must be
non-negative.\index{realizable!non-negative condition}
\item (Congruence) The sum over divisors must be
divisible by~$n$ for each~$n$.\index{realizable!congruence condition}
\end{itemize}
Thus a realizable sequence is a Dold sequence
in the sense of
Definition~\ref{definitionDoldSequence},
together with a sign condition on the signed
linear combinations in~\eqref{equationMainCondition}.


Figure~\ref{figtwo} illustrates this with
(the start of) a construction
of a map with the property that
the number of closed orbits of
length~$n$ is~$n$ for each~$n\in\N$.
The set~$X$ consists of the black dots and
the map~$T$ is given by the cyclic motions illustrated
by the arrows
(and of course the figure is to be extended to
the right in a similar way).

\begin{figure}[htbp]
      \begin{center}
\scalebox{0.58}{\includegraphics{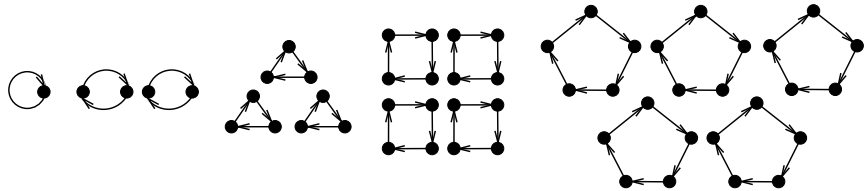}}
\caption{\label{figtwo}Building a map with~$n$ closed
orbits of length~$n$ for each~$n$.}
\end{center}
\end{figure}

\subsection{Relating Dold and Realizable Sequences}
\label{sectionRelatingDoldAndRealizableSequences}

For specific classes of dynamical systems the
relation between the Dold condition and realizability may
be quite involved, but at the level of full generality
the relationship is straightforward---and explains the
sense in which the realizable sequences may be
viewed as a positive cone in the space of Dold
sequences. The next result is shown in many places,
including work of
Du, Huang, and Li~\cite{MR1950443}
and Arias de Reyna~\cite{MR2163516}.

\begin{proposition}\label{propositionDifferencesBetweenRealizableSequences}
The following conditions are equivalent for an
integer sequence:
\begin{enumerate}[\rm(a)]
\item $(a_n)$ is a Dold sequence;
\item $(a_n)$ can be written as a difference of
two realizable sequences; and
\item\label{propositionDifferencesBetweenRealizableSequences3} for every prime~$p$ and~$n,\ell\in\N$ we have~$a_{np^{\ell}}\equiv
a_{np^{\ell-1}}$ modulo~$p^{\ell}$.
\end{enumerate}
\end{proposition}

\begin{proof}
The defining relations~\eqref{equationMainCondition}
and~\eqref{equationDoldSequenceMobiusCondition}
make it clear that~(b) implies~(a).
Assume now that~(a) holds, and let~$b_n=\sum_{d\smalldivides n}\mu\(\frac{n}{d}\)a_d$,
so we know that~$n\divides b_n$ for all~$n\ge1$.
Let~$b_n^{\pm}=\max(0,\pm b_n)$,
so that~$b_n^{+}$ and~$b_n^{-}$ are
non-negative integers with
\[
b_n=b_n^{+}-b_n^{-}
\]
for all~$n\ge1$.
Correspondingly, define~$a_n^{\pm}=\sum_{d\smalldivides n}db_d^{\pm}$.
Then~$(a_n^{+})$ and~$(a_n^{-})$ are realizable
by~\eqref{equationMainCondition}
and~$a_n=a_n^{+}-a_n^{-}$ for all~$n\ge1$.

To see that~(b) implies~(c), it is enough to show~(c)
for realizable sequences because the congruence
condition is closed under subtraction of sequences.
So assume that~$(a_n)$ is realizable
by some system~$(X,T)$. To show~(c)
it is enough to prove that~$a_{np^{\ell}}\equiv
a_{np^{\ell-1}}$ modulo~$p^{\ell}$
under the additional assumption
that~$p\notdivides n$.
By~\eqref{equationfixintermsoforbits}
we have
\[
\fix_T(p^{\ell}n)-\fix_T(p^{\ell-1}n)
=
\sum_{d\smalldivides p^{\ell}n}d\orbit_T(d)-
\sum_{d\smalldivides p^{\ell-1}n}d\orbit_T(d)
\]
which
must be divisible by~$p^{\ell}$ as any~$d$
with~$d\divides p^{\ell}n$ but~$d\notdivides p^{\ell-1}n$
has~$p^{\ell}\divides d$, showing~(c).

Finally, assume that~$(a_n)$ satisfies~(c).
Notice that in order to show that~$(a_n)$
is a Dold sequence it is enough to prove
that
\begin{equation}\label{equationMadMed1}
\sum_{d\smalldivides n}\mu\(\frac{n}{d}\)a_d\equiv0
\end{equation}
modulo~$p^{\ell}$
for any prime~$p$ with~$p^{\ell}\Divides n$\label{pageStrongDivides}
(meaning that~$p^{\ell}\divides n$ but~$p^{\ell+1}\notdivides n$).
So assume that~$p^{\ell}\Divides n$ and~$n=p^{\ell}m$.
Then in the sum on the left-hand
side of~\eqref{equationMadMed1}
only the terms with~$p^{\ell-1}\divides d$
do not vanish. Thus, by multiplicativity
of the M{\"o}bius function,
we get
\begin{align*}
\sum_{d\smalldivides n}\mu\(\frac{n}{d}\)a_d
&\equiv
\sum_{d\smalldivides m}\mu\(\frac{m}{d}\)a_{p^{\ell}d}
-
\sum_{d\smalldivides m}\mu\(\frac{m}{d}\)a_{p^{\ell-1}d}\\
&\equiv
\sum_{d\smalldivides m}\mu\(\frac{m}{d}\)
\bigl(a_{p^{\ell}d}-a_{p^{\ell-1}d}\bigr)
\equiv0
\end{align*}
modulo~$p^{\ell}$,
showing~(a).
\end{proof}

Hence, for example, we can easily see from property \ref{propositionDifferencesBetweenRealizableSequences3} in Proposition~\ref{propositionDifferencesBetweenRealizableSequences}  that the class
of Dold sequences is closed under taking sums, differences,
and products.

\subsection{Linear Recurrence Realizable Sequences}
\label{sectionDynamicalLinearRecurrenceSequences}

Linear recurrence sequences form a
natural class of well-studied sequences
(we refer to~\cite{MR1990179} for an
overview of some of their properties
and an extensive bibliography).
The first natural direction in which to
extend Fermat's little theorem is to integer
matrices.

\begin{lemma}\label{lemmaTracesAreRealisable}
If~$A\in\matrices_{d,d}(\N_0)$
is a matrix with non-negative integer
entries, then the sequence~$\(\trace(A^n)\)$ is
realizable.
\end{lemma}

\begin{proof}
Consider a directed graph~$\mathsf{G}=(V,E)$
comprising a
vertex set
\[
V=\{v_1,v_2,\dots,v_d\}
\]
and a set~$E$ of directed edges
between elements of~$V$
with~$a_{i,j}\in\N_0$ directed edges from the
vertex~$v_i$ to the vertex~$v_j$ for~$i$
and~$j$
in~$\{1,\dots,d\}$.
Writing~$A=(a_{i,j})$ for the resulting
adjacency matrix, the matrix~$A^n=\bigl(a_{i,j}^{(n)}\bigr)$
has~$(i,j)$th entry~$a_{i,j}^{(n)}$ equal to the
number of walks in the graph~$\mathsf{G}$ of
length~$n$ starting at vertex~$v_i$ and ending
at vertex~$v_j$. In particular,~$\trace A^n$ is
thus equal to the number of walks of length~$n$
that begin and end at the same vertex.
If we let~$X\subset E^{\Z}$ be the set of
bi-infinite walks~$(e_i)$ in the graph~$\mathsf{G}$,
then~$X$ is a compact set in the topology
inherited from the Tychonoff topology on~$E^{\Z}$,
and the left shift map~$\sigma\colon X\to X$
defined by~$\sigma\((e_i)\)=(e_{i+1})$ is a homeomorphism.
Moreover,~$\fix_{\sigma}(n)=\trace A^n$
for all~$n\ge1$ since there is a one-to-one
correspondence between sequences fixed by~$\sigma^n$
and walks of length~$n$ that start and end at
the same vertex.
\end{proof}

\begin{corollary}\label{croollaryWhereYouTheLastPersionToKnow}
Let $B$ be a square matrix with integer entries. Then the sequence~$\(\trace(B^n)\)$
is a Dold sequence.
\end{corollary}

\begin{proof}
Fix~$n$ and choose a non-negative matrix~$A$ with~$A\equiv B$ modulo~$n$.
Then the condition~\eqref{equationDoldSequenceMobiusCondition}
for~$A$ holds by Lemma~\ref{lemmaTracesAreRealisable},
and implies the condition~\eqref{equationDoldSequenceMobiusCondition}
for~$B$,
since all the expressions involved only involve operations
that commute with reduction modulo~$n$.
\end{proof}

We can further specialize this result to deduce congruences
of Euler--Fermat type.\index{Euler--Fermat congruence}

\begin{corollary}\label{matrixeulerfermat}
Let~$A$ denote a square matrix with integer entries.
Then
\[
\trace(A^{p^r})\equiv\trace(A^{p^{r-1}})
\]
modulo~$p^r$ for any prime~$p$ and~$r\in\N$.
\end{corollary}

\begin{proof}
This is the Dold congruence condition~\eqref{equationDoldSequenceMobiusCondition} of the sequence~$(\trace(A^n))$ applied to~$n$ being a prime power.
\end{proof}


This kind of reasoning may be used to find
non-trivial congruences; we refer to
work of Arias de Reyna~\cite{MR2163516} and
Sun~\cite{MR2887613} for examples.

Minton's result in Theorem~\ref{theoremMinton}
gives a complete description of linear recurrence
Dold sequences. This gives partial information about
realizable linear recurrence sequences, but not
complete information because of the sign conditions
and possible rational multiplier.
To see how this looks in a simple setting, we
describe in Example~\ref{exampleYash1}
one of the early results in this direction,
where a complete picture is readily found.

{Sign conditions for linear recurrence sequences
present difficulties in several directions.
There are many ways to see that the absolute
value of a linear recurrence sequence is not in general
a linear recurrence; the next example was shown to
us by Jason Bell.

\begin{example}\label{exampleBell1}
Let~$a_n=1+n\imag$ for all~$n\ge1$,
so that the
sequence~$(a_n)$
satisfies the
linear recurrence
relation~$a_{n+2}=2a_{n+1}-a_n$ for all~$n\ge1$.
Then we have~$|a_n|=\sqrt{1+n^2}$,
so the sequence~$(\vert a_n\vert)$
cannot satisfy a linear
recurrence relation
because the terms of a linear recurrence sequence
must lie in a finitely generated field extension
of the rationals and
the extension~$\Q(\sqrt{1+n^2}\mid n\in\N)/\Q$
is not finitely generated.
\end{example}

Clearly, if a sequence has constant sign, then taking
its absolute value preserves the property of being a linear
recurrence sequence. The same holds under the weaker assumption
that the signs change periodically. As it turns out, for real-valued sequences this condition is also necessary.

\begin{lemma}\label{lemabslr}
Let~$(a_n)$ be a linear recurrence sequence with real values.
Then~$(|a_n|)$ is also a linear recurrence sequence
if and only if the sequence of signs~$(\mathrm{sgn}(a_n))$ with
values in~$\{-1,0,1\}$ is {ultimately} periodic.
\end{lemma}
\begin{proof}
Assume that~$(|a_n|)$
is a linear recurrence, and consider the sets
\[
J_-=\{n\in\N \mid a_n <0\},\qquad J=\{n\in\N \mid a_n =0\},\qquad J_+=\{n\in\N \mid a_n > 0\}.
\]
For~$n\not\in J$, the sequence~$\mathrm{sgn}(a_n)=|a_n|/a_n$
can be written as a quotient of two linear recurrence sequences.
Since this sequence takes values~$\pm 1$,
we may apply the
Hadamard quotient theorem~\cite[Th.~4.4]{MR1990179}
to see that there is a linear
recurrence sequence~$(b_n)$ with~{$\mathrm{sgn}(a_n)=b_n$}
for all~$n\notin J$.
By the Skolem--Mahler--Lech theorem~\cite[Th.~2.1]{MR1990179},
the sets~$J_-$,~$J$, and~$J_+$ are unions of finite sets
and finitely many arithmetic progressions. It follows that~$(\mathrm{sgn}(a_n))$ is {ultimately} periodic\end{proof}

The problem of determining
if a given linear recurrence sequence of algebraic
numbers has positive
terms matters in multiple fields, and
the decidability and complexity (in the sense
of computer science) of this question
has a long history. Mignotte
{\it{et al.}}~\cite{MR743965}
and Vereshchagin~\cite{MR808885}
independently used Diophantine analysis to show
the problem is decidable for linear recurrence
sequences of degree no more than~$4$, and
Ouaknine and Worrell~\cite{MR3238382} brought in additional
tools from number theory and real algebraic
geometry to prove decidability and establish
the complexity of the question for
linear recurrence sequences over the
integers of degree
up to~$9$, under the assumption that there
are no repeated
characteristic roots. We refer to their
paper~\cite{MR3238382} for more of the history,
and for additional references.}

\begin{example}[Fibonacci recurrences]\label{exampleYash1}
Consider the sequence
\[
(1,a,1+a,1+2a,2+3a,3+5a,\dots)
\]
for~$a\in\N_0$. For which values of~$a$ is
this realizable?
\begin{itemize}
\item If~$a=0$ the sequence begins~$(1,0,1,1,2,3,5,\dots)$.
If this is realised by some~$(X,T)$ then~$\orbit_T(1)=1$
(because there is one fixed point), which means
that~$\fix_T(2)=\orbit_T(1)+2\orbit_T(2)=0$ is impossible.
\item If~$a=1$ the sequence begins~$(1,1,2,\dots)$
which is impossible as~$\fix_T(3)-\fix_T(1)$ must
be divisible by $3$.
\end{itemize}
\end{example}

{Theorem~\ref{theoremMinton}
shows that the space of all Dold solutions
of this recurrence has rank one, which means
the space of realizable solutions must have
rank zero or one. In fact it has rank one,
and there is a complete description.
Of course this result is an easy consequence
of Minton's Theorem~\ref{theoremMinton}, because
in this case the sign issues are easy,
but we include a proof to illuminate the sort
of issues that come into play in a simple
example.}

\begin{lemma}[Puri {\&} Ward~\cite{MR1866354}]\label{theoremFibonacciCondition}
The sequence
\begin{equation}\label{equationUsingaToTestForRealisability}
u=(u_n)=(1,a,1+a,1+2a,2+3a,3+5a,\dots)
\end{equation}
is realizable if and only if~$a=3$.
\end{lemma}

\begin{proof}
Lemma~\ref{lemmaTracesAreRealisable} applied to
the matrix
\[
A=\bmatrix1&1\\1&0\endbmatrix
\]
gives a system~$(X,T)$ with~$\fix_T(n)=L_n$
for~$n\ge1$, where~$L_n$\label{pageLucasNumbers}
denotes the~$n$th
term of the Lucas sequence~$(1,3,4,7,\dots)$.\index{Lucas sequence}
Thus the sequence~\eqref{equationUsingaToTestForRealisability}
is indeed realizable if~$a=3$.

For the converse direction,
write~$F=(F_n)=(1,1,2,3,5,8,\dots)$
for the Fibonnaci numbers.\index{Fibonacci number}\label{pageFibonacciNumbers}
Since~$(L_n)$ is realizable,
we have
\[
\sum_{d\divides n}\mu\(\frac{n}{d}\)L_d\equiv0
\]
modulo~$n$
for all~$n\ge1$.
For~$n=p$ a prime, it follows
that
\begin{equation}\label{equationCoolLucasThing}
L_p
=
F_{p-2}+3F_{p-1}
\equiv L_1=1
\end{equation}
modulo~$p$ (the first
equality is an easy induction).
It follows from~\eqref{equationCoolLucasThing}
that
\begin{equation}\label{equationSecondCoolLucasThing}
F_{p-1}\equiv1
\Longleftrightarrow
F_{p-2}\equiv-2
\end{equation}
modulo~$p$ for any prime~$p$.
The Fibonacci numbers
satisfy~$F_{p+1}\equiv0$ modulo~$p$
if~$p$ is a prime congruent to~$\pm2$ modulo~$5$
(see~\cite[Th.~180]{MR568909}).
Let~$p$ be any such prime.
The resulting identities~$
F_{p+1}
=
F_{p}+F_{p-1}
=
2F_{p-1}+F_{p-2}
\equiv 0
$
modulo~$p$
and~\eqref{equationCoolLucasThing}
together imply that~$F_{p-1}\equiv1$ modulo~$p$
if~$p\equiv\pm2$ modulo~$5$.

Assume now that~\eqref{equationUsingaToTestForRealisability}
is realizable. Then~$u_p\equiv u_1=1$
modulo~$p$ for any prime~$p$.
Now~$u_n=F_{n-2}+aF_{n-1}$
for~$n\ge1$, so~$F_{p-2}+aF_{p-1}\equiv 1$
modulo~$p$.
If~$p\equiv\pm2$ modulo~$5$
we thus have
\begin{equation}\label{eight}
\left(F_{p-2}-1\right)+a\equiv0
\end{equation}
modulo~$p$, since~$F_{p-1}\equiv1$.
For such~$p$,~\eqref{equationSecondCoolLucasThing}
gives~$F_{p-2}\equiv -2$ modulo~$p$,
so~\eqref{eight}
gives~$a\equiv 3$
modulo~$p$.
By Dirichlet's theorem\index{Dirichlet's theorem}
there are infinitely many primes~$p$
congruent to~$\pm2$ modulo~$5$, so~$a=3$ as claimed.
\end{proof}





Extending this line of reasoning gives the following
picture. For a given (homogeneous) linear recurrence\index{linear recurrence!relation}
relation there will be some
initial conditions that give realizable
sequences. There is always one such
sequence because~$(0,0,0,\dots)$ is
realizable. Since any natural number multiple
of a realizable sequence is also
realizable (by taking disjoint unions
of a given system with itself), we can
ask for the `rank' or `dimension' of the
realizable subspace, meaning the
number of linearly independent realizable
solutions.\index{realizable!subspace}
Lemma~\ref{theoremFibonacciCondition}
shows that there is a one-dimensional
subset of realizable sequences satisfying
the Fibonacci recurrence relation.

A more involved argument along the
same lines using special primes gives
the following result,
originally shown by Puri~\cite{yash}
using other methods.
{Once again Theorem~\ref{theoremMinton} gives an
upper bound on the dimensions, but other arguments
are needed to give exact values or lower bounds.}

\begin{theorem}[Everest {\it{et al.}}~\protect{\cite[Th.~2.1]{MR1938222}}]\label{theoremEPPWunprivedhere}
Let~$\Delta$ denote the discriminant of the
characteristic polynomial\index{linear recurrence!relation}\index{linear recurrence!non-degenerate}
of a non-degenerate binary recurrence relation.
Then the realizable
subspace has
\begin{enumerate}[\rm(a)]
\item dimension~$0$ if~$\Delta <0$,
\item dimension~$1$ if~$\Delta=0$ or~$\Delta > 0$
and non-square,
\item dimension~$2$ if~$\Delta >0$ is a square.
\end{enumerate}
\end{theorem}

The
fact that realizable sequences can be
multiplied by natural numbers and added
while remaining realizable
gives the following illustration of
case~(c):
any sequence~$u_n=a2^n+b$
with~$a,b\in\N$ is realizable and
satisfies
the recurrence relation~$u_{n+2}=3u_{n+1}-2u_n$
for all~$n\in\N$.

The proofs of the definitive
Lemma~\ref{theoremFibonacciCondition}
and Theorem~\ref{theoremEPPWunprivedhere}
rely on special properties of quadratics.
The general picture in higher degree
is much less clear, and new phenomena arise.
A partial result with a slightly more
permissive definition is found by
Everest {\it{et al.}}

\begin{theorem}[Everest {\it{et al.}}~\protect{\cite[Th.~2.3]{MR1938222}}]\label{rankrestriction}
Let~$f$ be the
characteristic polynomial\index{linear recurrence!characteristic polynomial}
of a
non-degenerate linear recurrence
sequence with integer coefficients.
If~$f$ is separable
with~$\ell$ irreducible
factors
and a dominant root, then the dimension of the
subspace
of solutions of the recurrence
{with the property that their absolute
value is realizable} cannot
be more than~$\ell$. If~$f(0)\neq 0$
then equality holds if either
the dominant
root is not less than the sum of the absolute values of
the other roots or
the dominant root
is strictly greater than the sum of the absolute values of
its conjugates.
\end{theorem}

Extending this to all linear recurrences,
and dealing with the sign conditions, remains open
{apart from the general result of
Minton in Theorem~\ref{theoremMinton}, which
gives an upper bound of the possible dimensions
for the original problem of the realizability of the
terms of a linear recurrence as opposed to the
absolute value of the terms as considered here}.
The Fibonacci case above only really used the
congruence condition in~\eqref{equationMainCondition},
but the general case also engages the
non-negativity condition, which can be
difficult to work with.
A simple example of a dominant root cubic
is given by the so-called
`Tribonacci' relation.

\begin{example}\label{tribonacci}
The sequence~$(3, 1, 3, 7, 11, 21, 39,\dots)$
satisfying
\[
u_{n+3}=u_{n+2}+u_{n+1}+u_n.
\]
is
realizable, since
its~$n$th term is~$\left(\trace(A^n)\right)$,
where~$A$ is the companion matrix to
the polynomial~$x^3-x^2-x-1$.
Theorem~\ref{rankrestriction} says that
any realizable sequence satisfying the same
cubic recurrence
must be a multiple
of this one.
\end{example}

In a rather different
direction, work of Kim, Ormes and Roush~\cite{MR1775737} on
the Spectral Conjecture of Boyle and Handelman~\cite{MR1097240}
gives a checkable criterion for a given linear recurrence sequence
to be realised {by an irreducible subshift of finite type}.
Bertrand-Mathis~\cite{MR2340599} discussed related questions
from a language-theoretic point of view.

The original proofs of
Theorem~\ref{theoremEPPWunprivedhere} and~\ref{rankrestriction}
again ultimately involve knowing there are
infinitely many primes satisfying certain conditions.
For the quadratic case the argument rests on
inert primes. The higher degree argument
in~\cite{MR1938222} also uses the result of
Vinogradov~\cite{zbMATH03029875} \index{Vinogradov theorem}
(which says that the sequence of
fractional parts of~$p\beta$
as~$p$ runs through the primes
is dense in~$(0,1)$ for irrational~$\beta$)
to ensure there are infinitely many
primes satisfying a sign condition
needed for the argument.

\subsection{Time-Changes Preserving Realizability}
\label{sectionTimeChangesPreservingRealizability}

The two viewpoints on realizable sequences---on the one hand, being
defined by a dynamical system~$(X,T)$ and, on the other,
being a sequence of non-negative integers satisfying
the combinatorial  condition~\eqref{theoremPhasnotorsion}---raises
questions about `functorial' properties. That is,
how to interpret operations that are natural in one setting
in the other setting. This idea is explored further in work
of Pakapongpun {\it{et al.}}~\cite{MR2486259,MR3194906,MR1873399}.
The very simplest of these, for example, is to note that changing~$T$ to~$T^2$
changes the sequence~$(\fix_T(n))$ to the sequence~$(\fix_T(2n))$.
We can interpret this observation by saying that the map~$n\mapsto2n$ for~$n\in\N$
is a realizability-preserving time change in the
following sense.

\begin{definition}[Jaidee {\it{et al.}}~\cite{MR4002553}]\label{definitionSymmetriesOfZetaFunctions}
For a system~$(X,T)$
define~$\mathscr{P}_T$,
the
set of \emph{realizability-preserving time-changes for}~$(X,T)$,
to be the
set of maps~$h\colon\mathbb{N}\to\mathbb{N}$
with the property that~$\bigl(\fix_T(h(n))\bigr)$
is a realizable sequence.
Also
define the monoid\label{pageMonoidP} of \emph{universally
realizability-preserving time-changes},~$\mathscr{P}=
\bigcap_{\{(X,T)\}}\mathscr{P}_T$,
where
the intersection is taken over all systems~$(X,T)$.
\end{definition}

Not only is~$\mathscr{P}$ a monoid---closed under composition---it
is in fact closed under infinite composition in the following sense.
If~$(h_1,h_2,\dots)$ is a sequence
of functions in~$\mathscr{P}$ with
the property
that
the sequence
\begin{equation}\label{equationwhoops!}
\(h_1(n),h_2(h_1(n)),h_3(h_2(h_1(n))),\dots\)
\end{equation}
stabilizes for every~$n\in\N$, then
the infinite composition~$h=\cdots\circ h_3\circ h_2\circ h_1$,
defined by setting~$h(n)$ to be the stabilized value
or limit of~\eqref{equationwhoops!},
is also in~$\mathscr{P}$.

\begin{theorem}[Jaidee {\it{et al.}}~\cite{MR4002553}]
The monoid~$\mathscr{P}$ has the following properties.
\begin{enumerate}[\rm(a)]
\item A polynomial lies in~$\mathscr{P}$ if and only if it is a monomial.
\item The monoid~$\mathscr{P}$ is uncountable.
\end{enumerate}
\end{theorem}

The fact that monomials lie in~$\mathscr{P}$ was shown first
in Moss' thesis~\cite{pm}, and of course means
that if a sequence~$(a_n)$ satisfies the Dold
congruence, then for any~$k,\ell\in\N$ the sequence~$(a_{\ell n^k})$
also does.
The proof that~$\mathscr{P}$ is uncountable relies
on a simpler remark also due to Moss: For a given prime~$p$,
the map~$g_p\colon\mathbb{N}\to\mathbb{N}$
defined by
\[
g_p(n)=\begin{cases}n&\mbox{if }p\notdivides n;\\pn&\mbox{if }p\divides n\end{cases}
\]
lies in~$\mathscr{P}$. Using the closure under
infinite composition property mentioned above, this
allows us to embed the power set of the set of primes
into~$\mathscr{P}$. One of the questions raised in~\cite{MR4002553}
asked if a nontrivial permutation could lie in~$\mathscr{P}$, and
we answer this here, in a stronger form.

\begin{theorem}\label{theoremPhasnotorsion}
The only surjective map~$\mathbb{N}\to\mathbb{N}$ in~$\mathscr{P}$
is the identity.
\end{theorem}

\begin{proof}
Let~$\sigma\colon\N\to\N$ be a surjective map,
assume that~$\sigma\in\mathscr{P}$, and fix~$k\in\N$.
Applying~$\sigma$ to~$\eper_k$ gives the
sequence~$(\eper_k(\sigma(n)))$, which by
definition of~$\eper_k$ can only take on the values~$0$ and~$k$,
and by surjectivity of~$\sigma$ must take on
the value~$k$.
By thinking about the possible closed orbits (or from
the Dold congruence), we know that a
realizable sequence only taking on
a single non-zero value~$k$ is necessarily an integer multiple of the sequence~$\eper_d$ for some~$d$ dividing~$k$. 
If~$d(k)$ is the smallest integer
with the property that~$k\divides\sigma(d(k))$,
it follows that~$d(k)\divides k$ and~$k\divides\sigma(n)$
if and only if~$d(k)\divides n$ for any~$n\in\N$.
That is,
\begin{equation}\label{equationMemoriesTamperedWith}
\sigma^{-1}(k\N)=d(k)\N.
\end{equation}
Let~$s=\Omega(k)$ be the number of prime divisors of~$k$ counted
with multiplicity, and
define the divisors~$m_0,m_1,\dots,m_s$ of~$k$
so that~$m_0=1$ and~$\frac{m_{j+1}}{m_j}$
is a prime for~$j=1,\dots,s-1$.
Then there are strict inclusions
\[
m_0\N\supsetneq m_1\N\supsetneq m_2\N\supsetneq\cdots\supsetneq m_s\N=k\N.
\]
Since~$\sigma$ is surjective, applying~$\sigma^{-1}$ gives
strict inclusions
\begin{equation}\label{equationMemoriesTamperedWith2}
d(m_0)\N\supsetneq d(m_1)\N\supset d(m_2)\N\supsetneq\cdots\supsetneq d(m_s)\N=d(k)\N
\end{equation}
by~\eqref{equationMemoriesTamperedWith}.
If~$d(k)$ is a proper divisor of~$k$ then~$\Omega(d(k))<\Omega(k)$,
contradicting~\eqref{equationMemoriesTamperedWith2}.
It follows that~$d(k)=k$ for all~$k\in\N$.
Now the definition of~$d(k)$
shows that~$k\divides\sigma(n)$ if and only if~$k\divides n$
for any~$k,n\in\N$.
Applying this
with~$k=n$ and~$k=\sigma(n)$
shows that~$\sigma(n)=n$ for all~$n\in\N$.
\end{proof}

\subsection{Repair Phenomenon}

As remarked earlier, the Fibonacci sequence~$(F_n)$
is not realizable.
Moss and Ward~\cite{mossfibo}
strengthened this observation, and
showed that the Fibonacci sequence sampled
along the squares can be `repaired' to be realizable in the
following sense.

\begin{theorem}[Moss and Ward~\cite{mossfibo}]
If~$j$ is odd, then
the set of primes dividing denominators
of~$\frac{1}{n}\sum_{d\smalldivides n}\mu\bigl(\tfrac{n}{d}\bigr)F_{d^j}$
for~$n\in\N$ is infinite.
If~$j$ is even, then the sequence~$(F_{n^j})$ is
not realizable, but the sequence~$(5F_{n^j})$ is.
\end{theorem}

This motivated the following definition.

\begin{definition}[Miska and Ward~\cite{miska2021stirling}]\label{definitionAlmostRealizable}
For a sequence~$A$ of non-negative integers,
\[
\failure(A)
=
\begin{cases}
\lcm
\(\{
\denom\((\tfrac1n\mu*A)(n)\)\mid n\ge1
\}\)&\mbox{if this is finite};\\
\infty&\mbox{if not}.
\end{cases}
\]
The sequence~$A$ is said to be \emph{almost realizable}
if~$\failure(A)<\infty$ and
it satisfies the sign condition.
\end{definition}

Thus the result of~\cite{mossfibo}
states in particular that~$(F_n)$ is not
almost realizable, but~$(F_{n^2})$ is
almost realizable with~$\failure((F_{n^2}))=5$.
This is part of a wider phenomenon in which
a Lucas sequence sampled along the squares
becomes realizable if it is multiplied by its
discriminant.

There is another combinatorial setting where
a similar almost realizability occurs.
Write~$\stirlingone{n}{k}$ for the (signless) Stirling
numbers of the first kind, defined for any~$n\ge1$ and~$0\le k\le n$
to be the number of permutations of~$\{1,\dots,n\}$
with exactly~$k$ cycles,
and write~$\stirlingtwo{n}{k}$ for~$n\ge1$ and~$1\le k\le n$
for the Stirling numbers of the second kind,
so~$\stirlingtwo{n}{k}$ counts the number of ways to partition
a set comprising~$n$ elements
into~$k$ non-empty subsets.

\begin{theorem}[Miska and Ward~\cite{miska2021stirling}]
For each~$k\ge1$ define sequences
\begin{align*}
\sequenceone_k
&=
\bigl(\stirlingone{n+k-1}{k}\bigr)_{n\ge1}
\intertext{and}
\sequencetwo_k
&=
\bigl(\stirlingtwo{n+k-1}{k}\bigr)_{n\ge1}.
\end{align*}
Then:
\begin{enumerate}[\rm(a)]
\item For~$k\ge1$ the sequence~$\sequenceone_k$ is
not almost realizable.
\item For~$k\le2$ the sequence~$\sequencetwo_k$ is realizable.
\item For~$k\ge3$ the sequence~$\sequencetwo_k$ is not
realizable, but is almost realizable
with~$\failure\bigl(\sequencetwo_k\bigr)\divides(k-1)!$
for all~$k\ge1$.
\end{enumerate}
\end{theorem}

Computing the repair factor~$\failure\bigl(\sequencetwo_k\bigr)$
is problematic, because {\it{a priori}} it involves
an unbounded calculation.
Calculations suggest that
\[
\failure\bigl(\sequencetwo_k\bigr)
=
\lcm
\(\{
\denom\((\tfrac1n\mu*A)(n)\)\mid1\le n\le k^{\sharp}
\}\)
\]
where~$k^{\sharp}$ is the
largest prime power less than~$k$,
but this is not proved.
What little
is known about the sequence
\[
\bigl(\failure\bigl(\sequencetwo_k\bigr)\bigr)_{k\ge1}
=(1, 1, 2, 6, 12, 60, 30, 210, 840, 2520, 1260,\dots)
\]
is recorded in the entries~A341617
and~A341991 of the Online Encyclopedia of Integer
Sequences~\cite{OEIS}.

\subsection{Continuous Maps}\label{sectionContinuousMaps}

A simple observation from~\cite{MR1873399} is that
an arbitrary sequence of closed orbit
counts~$(o_n)$, when viewed
as a set, can be given a compact
topology for which the map is a
homeomorphism. This is easiest to see if there
is a fixed point as this can be placed `at
infinity'. Returning to Figure~\ref{figtwo},
we can simply locate the fixed point at the origin,
squeeze the closed orbits into some
bounded subset of~$\R^2$, and define the compact
space to be the union of the orbits as illustrated
in Figure~\ref{figthree}, to produce
a continuous map with (in this case)~$n$ closed
orbits of length~$n$ for each~$n\in\N$.

\begin{figure}[htbp]
      \begin{center}
\scalebox{0.58}{\includegraphics{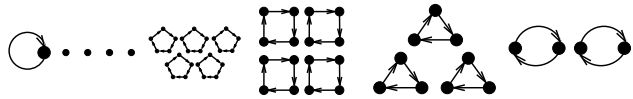}}
\caption{\label{figthree}Building a
continuous map with~$n$ closed
orbits of length~$n$ for each~$n$.}
\end{center}
\end{figure}

\begin{lemma}\label{lemmafirstcontinuousconstruction1}
A sequence~$(a_n)$ with~$a_n\in\N_0$ is
realised by a homeomorphism of a compact
metric space if and only if it
satisfies~\eqref{equationMainCondition}.
\end{lemma}

Even at this level of generality one has to respect
topological constraints. The spaces constructed above
are zero-dimensional, which permits
the extreme flexibility in the number of orbits.
On a specific topological space there are additional
conditions forced onto the sequence by
the global topology, {which will be discussed further
in Section~\ref{Indseq}}. Two examples
of this will illustrate some of what may
arise.

\begin{example}
A continuous map on a disk must have a fixed point by
Brouwer's theorem.
\end{example}

\begin{example}
A continuous map~$f\colon[0,1]\to[0,1]$
must obey\index{Sarkovskii@\v{S}arkovs\cprime ki\u{\i}'s theorem}
\v{S}arkovs\cprime ki\u{\i}'s theorem~\cite{MR0159905}.
This shows there is a total order~$\prec$
on~$\N$ with the property that if~$m\prec n$
then the existence of a point of {minimal} period~$m$
for~$f$ implies the existence of a point
of {minimal} period~$n$ for~$f$, and that this
order has the following shape:
\[
3\prec
5\prec 7\prec\cdots\prec 3\cdot 2\prec 5\cdot 2\prec\cdots\prec 3\cdot
2^2\prec 5\cdot 2^2\prec\cdots\prec 2^3\prec 2^2\prec 2\prec 1.
\]
\end{example}

\subsection{Smooth Maps}
\label{sectionSmoothMaps}

It is less clear how to create
smooth models. The following result
requires a delicate construction, \index{smooth model}
and this was done by Windsor.

\begin{theorem}[Windsor~\cite{MR2422026}]\label{theoremWindsor}
A sequence~$(a_n)$ with~$a_n\in\N_0$ is
realised by a~$C^{\infty}$
map of the~$2$-torus if and only if it
satisfies~\eqref{equationMainCondition}.
\end{theorem}

The start of this construction is straighforward ---
to build closed orbits of given length using rotations.
The challenge is to ensure that no new
orbits are created in the limiting processes
involved.

\begin{problem}[Smooth models in general]
{Is it true that for every smooth manifold~$M$
of dimension at least~$2$ a
sequence~$(a_n)$ with~$a_n\in\N_0$ is
realised by a~$C^{\infty}$
map of~$M$}
if and only if it
satisfies~\eqref{equationMainCondition}
and the consequences of the
Lefschetz fixed point theorem?\index{Lefschetz fixed point theorem}
\end{problem}

\section{Lefschetz Sequences}
\label{sectionLefschetzSequences}

\begin{definition}\label{definitionLefschetzSequence}
A sequence of integers~$(a_n)$ is
a \emph{Lefschetz sequence}\index{Lefschetz sequence}\index{sequence!Lefschetz}
if there exist
square integer matrices~$A$ and~$B$
such that
\[
a_n=\trace A^n-\trace B^n
\]
for all~$n\in\N$.
\end{definition}

Expressing the trace in terms of the eigenvalues
of the matrices gives the following
characterization.

\begin{proposition}\label{formLefschetz}
A sequence of integers~$(a_n)$
is a Lefschetz sequence if and only if
there exist algebraic
numbers~$\lambda_1,\dots,\lambda_s$
and integers~$m_1,\dots,m_s$
such that
\[
a_n=\sum_{i=1}^s m_i \lambda_i^n
\]
for all~$n\in\N$.
\end{proposition}

It is clear from
Corollary~\ref{croollaryWhereYouTheLastPersionToKnow}
that a Lefschetz sequence is a Dold sequence.

We recall that a sequence~$(c_n)$ (taking
values in a field) is a \emph{linear recurrence
sequence} if there exists
an integer~$p$
(the \emph{order} of the recurrence)
and constants
\[
\alpha_0,\ldots,\alpha_{p-1}
\]
{with~$\alpha_0\neq0$}
such that
\[
c_{n+p}=\alpha_{p-1} c_{n+p-1} +\ldots +\alpha_0 c_n
\]
for all~$n\geq 1$.
In applications,
it may happen that the recurrence is only satisfied
for~$n\geq n_0$, in which case the sequence nevertheless
satisfies a linear
recurrence with~$p'=p+n_0-1$,~$\alpha'_i = \alpha_{i-n_0}$
for~$i\geq n_0$ and~$\alpha'_i=0$ otherwise. The following lemma is
well-known (it may be found, for example
in the monograph of Salem~\cite{MR0157941}; we
also refer to~\cite{MR1990179} for results
of this sort and an extensive guide to the literature
on recurrence sequences).

\begin{lemma}\label{Salem1}
A power series
\begin{equation}\label{series}
f(z)=\displaystyle\mathop{\sum}_{n=0}^{\infty}c_nz^n
\end{equation}
with {coefficients in a field $\K$}
represents a rational function~$\frac{P(z)}{Q(z)}$
(where~$P$ and~$Q$ are polynomials) {if and only if}
its coefficients satisfy a linear recurrence relation.
Furthermore, the order of the recurrence is at most~$\max(\deg P, \deg Q)$.
\end{lemma}

It is clear that Lefschetz sequences constitute a proper subset of
the set of all Dold sequences.
There are many ways to see this, including the following.
\begin{itemize}
\item {\bf{Construction:}} There are dynamical systems~$(X,T)$
with the property that~$\sum_{n=1}^{\infty}
\fix_T(n)z^n$ is far from rational, each
of which gives a Dold sequence that is not Lefschetz.
\item {\bf{Symmetry:}} As
discussed in Section~\ref{sectionTimeChangesPreservingRealizability},
there is a notion
of `symmetry' in the space of all zeta
functions that shows, for example,
that if~$(a_n)$ is a realizable sequence
then~$(a_{n^2})$ is also a realizable sequence.
These are `symmetries' of the space of realizable
(and hence of Dold) sequences that clearly do not
preserve the Lefschetz property
(see Jaidee {\it{et al.}}~\cite{MR4002553}).
\item {\bf{Cardinality:}}
There are only countably many sequences that arise as
the difference of the traces of powers of integer matrices, but
there are uncountably many Dold sequences.
\end{itemize}

\begin{theorem}\label{Lefschetz}
A sequence of integers~$(a_n)$ is a Lefschetz
sequence
if and only if~$(a_n)$ is a Dold sequence
with the property that its
generating sequence~$(c_n)$ is a linear
recurrence sequence.\end{theorem}

\begin{proof}
Assume first that~$(a_n)$ is a Lefschetz sequence
(and hence a Dold
sequence),
so that we may write~$a_n=\sum_{i=1}^s m_i
\lambda_i^n$ for all~$n\ge1$.
From Theorem \ref{theoremEulertransforms}
we can calculate

\begin{align*}
\Bigl(1-\sum_{n\geq 1} c_n z^n\Bigr)
&=
\exp\Bigl(-\sum_{n\geq 1}\frac{a_n}{n}z^n\Bigr)
=
\prod_{i=1}^s\exp\Bigl(-m_i\sum_{n\geq 1}\frac{\lambda_i^n}{n}z^n\Bigr)\\ &=
\prod_{i=1}^s
\exp\bigl(m_i\log(1-\lambda_i z)\bigr)
=
\prod_{i=1}^s\bigl(1-\lambda_i z\bigr)^{m_i},
\end{align*}
where the last three equalities are formal.
We
deduce
that the generating function~$\sum_{n\geq 1}c_n z^n$
is rational,
so
by Lemma~\ref{Salem1} the sequence~$(c_n)$ satisfies a linear
recurrence relation.

This reasoning can be inverted:
if the generating sequence~$(c_n)$ is linear recurrent, then by Theorem~\ref{Salem1} the series $1-\sum_{n\geq 1} c_n z^n$ represents a
rational function~$\frac{P(z)}{Q(z)}$ (where~$P$ and~$Q$ are polynomials with rational coefficients), and since its constant term is~$1$,
we may factorize it over the complex numbers as
\[
1-\sum_{n\geq 1} c_n z^n = \prod_{i=1}^s\bigl(1-\lambda_i z\bigr)^{m_i}
\]
for complex numbers~$\lambda_i$ and integers~$m_i$. Repeating the previous formal computations gives~$a_n=\sum_{i=1}^s m_i \lambda_i^n$. Since the numbers~$\lambda_i$ arose as roots of~$P$ and~$Q$, they are algebraic numbers, and hence~$(a_n)$ is a
Lefschetz sequence by Proposition~\ref{formLefschetz}.
\end{proof}

\subsection{Generating sequences for Lefschetz sequences}

Generating sequences of Lefschetz numbers were
considered by Graff {\it{et al.}}~\cite{MR3917044},
and a certain characterization of generating sequences
of Lefschetz numbers of iterations was given. Here we will describe a
simple computational criterion for verifying that a given
sequence~$(c_n)$ is not a generating sequence of a Lefschetz
sequence (provided the dimensions of  respective matrices are
bounded from above).
{We start by recalling a well-known
test for rationality in terms of
the \emph{Kronecker--Hankel determinants}.
We refer to Salem~\cite{MR0157941}
or Koblitz~\cite{MR0466081} for convenient proofs.}

\begin{lemma}\label{Salem2}
A power series~$\sum_{n\geq 0}c_n x^n$ represents a rational
function if and only if the Kronecker--Hankel determinants
\begin{equation}\label{determinant}
\Delta_m=\det
\begin{pmatrix}
c_0&c_1&\dots&c_m
\\c_1&c_2&\dots&c_{m+1}
\\\vdots&\vdots&\ddots&\vdots
\\c_m&c_{m+1}&\dots&c_{2m}
\end{pmatrix}
\end{equation}
are all zero for~$m$ large enough.
{In fact~$(c_n)$ satisfies a linear
recurrence of order~$p$ if and only
if~$\Delta_m$ vanishes for~$m\ge p$.}
\end{lemma}

\begin{remark}\label{Hankel}
The determinant~\eqref{determinant} in Lemma~\ref{Salem2} is an
example of a Hankel determinant;
that is, a determinant of a square
matrix in which each ascending anti-diagonal is constant.
\end{remark}

\begin{proof}[Proof of Lemma~\ref{Salem2}]
We will only use---and
so will only prove---one implication, namely the fact that
if the series is rational
then its Hankel determinants of sufficiently high order
must vanish.
The relation
\[
c_{n+p}=\alpha_{p-1}c_{n+p-1}+\dots+\alpha_0 c_n
\]
shows that the~$(p+1)$-st column of the Hankel matrix
is a linear combination of the preceding~$p$ columns.
\end{proof}

\begin{theorem}
Let~$(a_n)$ be a sequence of Lefschetz numbers of
the form
\[
a_n=\trace A^n-\trace B^n
\]
for all~$n\in\N$,
where~$A\in\matrices_{k,k}(\C)$ and~$B\in\matrices_{\ell,\ell}(\C)$.
If~$(c_n)$ is the generating sequence of the sequence~$(a_n)$,
then the determinant of a Hankel
matrix~$\Delta_m$ vanishes for~$m\ge\max(k,\ell)$.
\end{theorem}

\begin{proof}
Let~$\lambda_i$
for~$1\leq i\leq k$ and~$\mu_j$
for~$1\leq j\leq\ell$ be the eigenvalues
of~$A$ and~$B$, respectively (with multiplicities).
Computing as in the proof of Theorem~\ref{Lefschetz}, we obtain \[
\exp\Big(-\sum_{n=1}^{\infty}{a_n\frac{x^n}{n}}\Big)
=
\frac{\prod_{j=1}^{\ell}(1-\mu_j z)}{\prod_{i=1}^k (1-\lambda_i z)}.
\]
By Lemma~\ref{Salem1}, the generating sequence satisfies a linear
recurrence of order
no more than~$\max(k,\ell)$,
and hence, by Lemma~\ref{Salem2},
the corresponding Hankel
determinants~$\Delta_m$ vanish for~$m\geq\max(k,\ell)$.
\end{proof}

\begin{example}
The sequence~$(1,3,2,4,5,7,6,8,9,\dots)$ is not a generating sequence of
a Lefschetz sequence obtained from matrices of
dimensions not more than~$4$
since
\[
\det
\begin{pmatrix}
1&3&2&4&5\\3&2&4&5&7\\2&4&5&7&6\\4&5&7&6&8\\5&7&6&8&9\end{pmatrix}
=
-256\neq0.
\]
\end{example}

\begin{lemma}\label{perDold-Lef}
Let~$(a_n)$ be a sequence of integers.
Then the following conditions are equivalent:
\begin{enumerate}[\rm(a)]
\item The sequence~$(a_n)$ is a periodic Lefschetz sequence.
\item The sequence~$(a_n)$ is a bounded Dold sequence.
\end{enumerate}
\end{lemma}

\begin{proof}
It is clear that a periodic Lefschetz sequence is a bounded Dold sequence.
By Lemma~\ref{lemmaDoldPeriodicIffBounded}, a bounded Dold sequence is a sum of integer multiples of elementary periodic sequences $\eper_k$. Any such sequence $\eper_k$ can be written as the sequence of traces of powers of matrices of the form
\begin{equation}\label{matrix} M=\begin{pmatrix}
0&0&\dots&0&1\\
1&0&\dots&0&0\\
0&1&\dots&0&0\\
\vdots&0&\ddots&0&\vdots\\
0&\dots&0&1&0
\end{pmatrix}_{k\times k}
\end{equation}
since~$\trace M^n=\eper_k(n)$. Thus any bounded Dold sequence
is a Lefschetz sequence.
\end{proof}

\subsection{Asymptotic Properties}
\label{sectionAsymptoticPropertiesLefschetzNumberIterations}

One may pose several questions concerning
possible growth rates of sequences that lie in more
than one of the various classes. For example, what
are the possible rates of growth of realizable sequences
that are linear recurrence sequences?
The relationship between growth and arithmetic
properties in a topological setting
for Lefschetz sequences is explored
in many places.
We cite below an interesting alternative
due to Babenko and Bogaty\u{\i}
(we also refer to the monograph of Jezierski and Marzantowicz~\cite{MR2189944}
for results in this direction).

\begin{theorem}[Babenko and Bogaty\u{\i},~\protect{\cite{MR1130026}}]
Let~$f$ be a map of a space
{with finitely-generated real homology spaces
(for example,
a map of a compact Euclidean neighbourhood
retract) and write~$\rm{sp_{es}}$ for the essential spectral
radius of the induced map in real homology.}
Then exactly one of the following three
possibilities holds:
\begin{enumerate}[\rm(a)]
\item $\lefschetz({f^n})=0$ for~$m=1,2,\dots$, which
happens if and only if~${\rm{sp_{es}}}(f)=0$.
\item The
sequence~$\(\frac{\lefschetz({f^n})}{{\rm{sp_{es}}}(f)^n}\)$
has the same
set of limit points as a periodic
sequence of the form~$\bigl(\sum_i{\alpha_i\epsilon_i^n}\bigr)$,
where~$\alpha_i\in\mathbb{Z}, \epsilon_i\in\mathbb{C}$,
and~$\epsilon_i^k=1$ for some~$k\in\N$.
\item The set of limit points of the
sequence~$\Bigl(\frac{|\lefschetz({f^n})|}{{\rm{sp_{es}}}(f)^n}\Bigr)$ contains an interval.
\end{enumerate}
\end{theorem}

\section[Topological Invariants of Iterated Maps as Dold Sequences]{Topological Invariants of Iterated Maps as Dold Sequences: Topological and Dynamical Consequences}\label{Indseq}

In this section we describe some dynamical and topological
consequences of the Dold congruences. In particular,
we show that they
are valid for many different topological invariants, and show how
the congruences may be transferred
into information
about the dynamical properties of maps, or
about the structure of
periodic points.

\subsection{Fixed Point Indices and
Applications}
\label{sectionDoldCongruencesFixedPointIndicesApplications}

Fixed point indices of iterations
were the original motivation for the
definition of Dold
sequences. We first
recall their definition and then
indicate
some consequences of the Dold congruences.

Consider a Euclidean neighbourhood retract~$Y$
and a continuous
map
\[
f\colon V\longrightarrow Y,
\]
where~$V \subset Y$ is an open subset,
and assume that~$\fixset_f(1)\subset V$ is compact.
Then there is a well-defined
\emph{fixed point index}~$\ind(f)=\ind(f,V)\in\mathbb Z$,\label{pageIndexOfAFixedPoint}
which is a topological invariant.
We refer to the monograph of
Jezierski and Marzantowicz~\cite[Sec.~2.2]{MR2189944}
for the formal definition.

\begin{definition}[Dold~\cite{MR724012}]\label{index-def}
We define the iterations~$f^n\colon V_n \to Y$
for~$V_n$ defined as follows.
We first set~$V_1=V$, and then
inductively define
\[
V_n=f^{-1}(V_{n-1})
\]
for~$n>1$.
Under the assumption that~$\fixset_{f}(n)$ is compact,
the fixed point index~$\ind(f^n)=\ind(f^n, V_n)$ is a
well-defined integer for each~$n\in\N$.
\end{definition}

Theorem~\ref{DoldTH} was
shown by Dold, although it was known
in certain cases earlier. In the terminology we
have adopted (which of course is not that used
by Dold) we have the following result.

\begin{theorem}[Dold~\protect{\cite[Th.~1.1]{MR724012}}]\label{DoldTH}
The sequence~$(\ind(f^n))_n$ is a Dold sequence.
\end{theorem}

The recent survey by Steinlein~\cite{MR3392979}
with an emphasis on topology
{and the earlier survey of Nussbaum~\cite{MR284888}
with an emphasis on non-linear functional analysis,}
contain
much of the interesting history of proofs
that~$(\ind(f^n))$ and~$(\lefschetz({f^n}))$ are
Dold sequences,
described in the more general context of
the Leray--Schauder degree.\index{Leray--Schauder degree}

\subsubsection{Existence of Broken Orbits}
\label{sectionExistenceOfBrokenOrbits}

\begin{definition}\label{broken}
Let~$f\colon\mathbb R^m\to\mathbb R^m$ be a continuous map
and let~$\Omega\subsetneq\mathbb R^m$ be
a bounded open set.
A periodic orbit with the property that
at least one of its points lies inside~$\Omega$ and
at least one of its points lies
outside of the closure~$\overline{\Omega}$
of~$\Omega$ will be called an\emph{~$\Omega$-broken
orbit}.\index{broken orbit}
\end{definition}

We define
\[
b_n(f,\Omega)=\frac{1}{n}\sum_{d\smalldivides n}
\mu(n/d)\ind(f^d,\Omega)
\]
for~$n\in\N$.

\begin{theorem}[Krasnosel\cprime ski\u{\i}
{\&} Zabre\u{\i}ko~\cite{MR736839};
Pokrovskii {\&} Rasskazov~\cite{MR2022383}]\label{broken-existence}
Let~$n$ be an integer, let
\[
f\colon\mathbb R^m\to\mathbb R^m
\]
be a continuous map with~$\partial\Omega\cap\fixset_f(n)=\emptyset$,
and assume that~$b_n(f,\Omega)\not\equiv 0$
modulo~$n$. Then there exists an~$\Omega$-broken orbit
whose minimal period is a divisor of~$n$.
\end{theorem}

\begin{proof}
Fix~$n$ and let~$G_k=\fixset_f(k)\cap\Omega$
for~$k\divides n$.
Assume for the
purposes of a contradiction
that for all~$k$
dividing~$n$ and for all~$x\in G_k$
we have~$f^i(x)\in
\Omega$ for~$1\le i\le k$.
Define~$V_n$ inductively, in the same way as in
Definition~\ref{index-def}, so
\[
V_1=\Omega,\dots,V_n=f^{-1}(V_{n-1}).
\]
Notice that
we then have
\[
V_n=\{x\in V\mid x,f(x),\dots,f^{n-1}(x)\in V\}
\supset G_k.
\]
By the localization property of\index{fixed point index!localization property}
the
fixed point index (see~\cite[Sec.~2.2.1]{MR2189944}),
we have~$\ind(f^n,\Omega)=\ind(f^n,V_n)$
and so~$b_n(f,\Omega)=b_n(f, V_n)$,
but by Theorem~\ref{DoldTH} we know that~$b_n(f,V_k)\equiv0$
modulo~$n$, which is a contradiction
to the assumption.
\end{proof}

We have stated Theorem~\ref{broken-existence}
for maps on~$\R^m$ because that is the context
considered by Pokrovskii {\&} Rasskazov~\cite{MR2022383},
who use this assumption to draw additional conclusions.
A similar proof should give the
result in the setting of a Euclidean
neighbourhood retract and compact
set of fixed points.

\subsubsection{Planar Homeomorphisms}
\label{sectionPlanarHomeomorphisms}

The Dold relations are useful in finding
restrictions on the form
of indices of iterations.

Let us recall that by a local fixed point index at an isolated fixed
point~$q$, written~$\ind(f,q)$, we understand
the index~$\ind(f,V)$ for a neighborhood~$V$ of~$q$
that is small enough to have~$\fix_f\cap V=\{q\}$.
The following theorem in this direction
was proved by Brown in~$1990$.

\begin{theorem}[Brown~\protect{\cite[Th.~4]{MR994772}}]\label{Brown}
Let~$f\colon\R^2\to\R^2$ be
a planar orientation preserving homeomorphism
{with an isolated fixed point at~$0$ for each iteration}.
Then
there is an integer~$p\neq1$ such that
\begin{equation}\label{Brown-homeo}
\ind(f^n,0)=
\begin{cases}
\ind(f,0)&\mbox{if }\ind(f,0)\neq 1;\\
1\mbox{ or }p&\mbox{if }\ind(f,0)=1
\end{cases}
\end{equation}
for all~$n\in\N$.
\end{theorem}

\begin{remark}
Brown conjectured that if~$\ind(f,0)=1$, then every integer~$p$
can appear as an index of some iteration in
the formula~\eqref{Brown-homeo} (cf.~\cite[Remark after Theorem~$4$]{MR994772}),
and gave examples of realizations for all
values of~$p$ except for~$p=0$ and~$p=2$.
The Dold congruences easily
exclude these two cases,
by showing that~$p=0$ and~$p=2$ cannot occur as
indices of any iteration if~$\ind(f,0)=1$ (cf. \cite{MR2037272}).
To see how this works, assume that~$p=0$ and
let~$n$ be the first iteration for
which~$\ind(f^n,0)=0$.
Then
\[
\sum_{k\smalldivides n}\mu(n/k)\ind(f^k,0)
=
\sum_{k\smalldivides n, k\neq n}\mu(n/k)
=
\sum_{k\smalldivides n}\mu(n/k)-1
=-1
\not
\equiv 0
\]
modulo~$n$,
where in the last equality we used well-known
identities for the M{\"o}bius function
contained in Lemma~\ref{CONVID}.
\end{remark}

Notice that
by the formula~\eqref{Brown-homeo}
and Lemma~\ref{lemmaDoldPeriodicIffBounded}
the sequence~$(\ind(f^n,0))$ must be periodic.

Let us recall that an isolated fixed point~$p$ is
non-accumulated if~${\rm Per}(f) \cap V =\{p\}$ for some neighborhood~$V$ of~$p$.
By a use of
subtle topological analysis,
Ruiz del Portal and Salazar showed
later in~\cite{MR2645114} that
if~$0$ is not an accumulated fixed point, then
\[
\ind(f^n,0)
=
\eper_1(n)+a_d\cdot\eper_d(n),
\]
where~$d\geq 1$ and~$a_d$ is an integer.

\subsubsection{Periodic Sequences of Indices of Iterations}
\label{sectionPeriodic Sequences of Indices of Iterations}

Consider a compact
Euclidean neighbourhood retract~$X$ and
continuous map~$f\colon X\to X$ such
that the two following conditions are satisfied:
\begin{enumerate}[\rm(a)]
\item the set~$\fixset_f(n)$ is compact for each~$n\ge1$ and
consists of isolated fixed points of~$f^n$;
\item for each~$x\in \Per(f)$, the set of periodic points of~$f$, the sequence~$(\ind(f^n, x))_n$ is
bounded.
\end{enumerate}

Notice that the condition~(a) is equivalent to the fact that
the number of~$n$-periodic point is finite
for each~$n\ge1$, while~(b)
means that~$(\ind(f^n, x))_n$ is periodic (by Lemma~\ref{lemmaDoldPeriodicIffBounded}).

In this class of maps, the fact that the Lefschetz numbers of
iterations of~$f$
are unbounded implies the existence of infinitely many periodic
points, by the following result.

\begin{theorem}\label{Lef-unbonded}
Let~$f\colon X \to X$ satisfy the  conditions~{\rm(a)} and~{\rm(b)}
above. If~$(\lefschetz({f^n}))$ is unbounded,
then~$f$ has infinitely many periodic
points with distinct periods.
\end{theorem}

\begin{proof}
By the Lefschetz--Hopf formula, we have
\begin{equation}\label{LH}
\lefschetz({f^n})
=
\sum_{x\in\fixset(f^n)}\ind(f^n,x)
\end{equation}
for each~$k\ge1$.

As a consequence of the formula~\eqref{LH},
the sequence~$(\lefschetz({f^n}))$ is a
sum of~$(\ind(f^n,x))$
over all~$x\in\bigcup_{n\ge1}\fixset_f(n)$.
By assumption~$(\lefschetz({f^n}))$ is unbounded,
while by~(b) each~$(\ind(f^n,x))_n$ is bounded and by~(a)
we know~$\fixset_f(n)$ is finite for each~$n\ge1$.
Thus~$f$ must have infinitely
many periodic points of distinct minimal periods.
\end{proof}

\begin{remark}\label{klasy}
One of the
directions of recent research is to identify the
classes of maps for which~$(\ind(f^n,x))$ is bounded.
Among such maps there are:
\begin{itemize}
\item planar maps and homeomorphisms
of~$\R^3$ at a fixed
point which is an isolated invariant set
(see
Hern\'{a}ndez-Corbato
and Ruiz del Portal~\cite{MR3343550},
Le Calvez, Ruiz del Portal, and Salazar~\cite{MR2739062},
and Le Calvez and Yoccoz~\cite{MR1477759});
\item $C^1$ maps in work of
Chow, Mallet-Paret and Yorke~\cite{MR730267};
\item holomorphic maps in work
of Bogaty\u{\i}~\cite{MR1095305},
Fagella and Llibre~\cite{MR1707699},
and Zhang~\cite{MR2465609};
\item simplicial maps of smooth type
in work of Graff~\cite{MR1874085}.
\end{itemize}
For such maps the structure of the indices of
iterations allows one to detect
information about periodic
points and dynamical behaviour in the neighborhood
of periodic points, and in certain cases
some features of the global dynamics as well.
The determination of the exact form of possible
indices of iterations  for smooth maps (see~\cite{MR2813888})
turned out to have many topological consequences.
In particular, based on that result
Graff and Jezierski constructed a smooth branch of
Nielsen periodic point theory, obtaining
invariants that allow the
minimal number of periodic points in a
smooth homotopy class to be computed~\cite{MR2754353, MR3071942, MR3622692}.
\end{remark}

\subsubsection{Detecting Periodic Points using Dold Congruences}
\label{sectionDetecting Periodic Points using Dold Congruences}

{This is again a large area of research, which we
illustrate with a sample of the type of result that may
be expected.}

\begin{proposition}[Dugundji and Granas~\cite{MR660439}]\label{4.1}
Let~$W$ be a connected
polyhedron\index{polyhedron}\index{homotopic to a constant}
and let~$f\colon W\rightarrow{W}$
be a continuous map with the property that~$f^n$ is
homotopic to a constant
map for some~$n\ge1$. Then~$f$ has a fixed point.
\end{proposition}

\begin{proof}
{Writing~$\sim$ for homotopy equivalence,
observe first that
\begin{equation*}\label{homotop}
f^n\sim c\Longrightarrow f^{n+1}\sim c,
\end{equation*}
where~$c$ is a constant map.}
In particular,~$f^p\sim c$ for any prime number~$p\ge n$.
Hence
\[
f_{*i}^p\colon H_i(W,\Q)\longrightarrow{H_i(W,\Q)}
\]
is the zero homomorphism for all~$i>0$,
so~$\lefschetz({f^p})=1$.
On the other hand,
\[
\lefschetz(f)\equiv\lefschetz({f^p})
\]
modulo~$p$ by the Dold congruences, which
implies that~$\lefschetz(f)\neq 0$.
This proves there must be a fixed point by
the Lefschetz fixed point theorem.
\end{proof}


\subsection{Nielsen and Reidemeister Numbers}
\label{sectionDold Congruences for Nielsen and Reidemeister Numbers}

Let~$K$ be a connected, compact polyhedron
with a continuous map~$f\colon K\to K$.
Let~$p\colon\widetilde{K}\to K$ be the universal
cover of~$K$,
and let~$\widetilde{f}\colon\widetilde{K}\to\widetilde{K}$ be a
lifting of~$f$, so~$p{\circ}\widetilde{f}=f{\circ}p$.
Liftings~$\widetilde{f}$ and~$\widetilde{f'}$
are said to be \emph{conjugate}\index{lifting}\index{lifting!conjugate}
if there is a~$\gamma\in\pi_1(K)$ with
\begin{equation}\label{conjugacy}
\widetilde{f'}
=
\gamma{\circ}\widetilde{f}{\circ}\gamma^{-1}.
\end{equation}
We call the subset~$p(\fixset(\tilde{f}))\subset\fixset(f)$ the
\emph{fixed point class}\index{fixed point class}
of~$f$ determined by the lifting class~$[\tilde{f}]$.
A fixed point class is called
\emph{essential}\index{fixed point class!essential}
if its fixed point index
is non-zero.
In this setting we can introduce the
\emph{Reidemeister number}~$R(f)$\index{Reidemeister number}\label{pageReidemeisterNumber}
of~$f$
and the \emph{Nielsen number}~$N(f)$\index{Nielsen number}\label{PageNielsenNumber}
of~$f$.
{In our setting,~$R(f)$ is the number of lifting classes
of~$f$, and~$N(f)$ is the number of essential fixed point classes.
Notice that~$R(f)$ is equal to the number of fixed point classes,
and  is a positive integer or infinity. Both~$R(f)$
and~$N(f)$ are topological invariants.}

{The importance of fixed point theory
in topology goes back to Poincar{\'e} and the
origins of topology itself.
Lefschetz~(see~\cite{MR1501331} and its
references to his earlier works in the area)
found a way to count fixed
points (with a multiplicity
given by the fixed point index) of continuous maps
on compact topological spaces in terms of traces of induced
maps on the homology groups of the space.
Nielsen~\cite{zbMATH02597378} studied the minimal
number of fixed points in an isotopy class
of homeomorphisms of the torus,
a result extended by Brouwer to continuous
maps of the torus, and went on to
publish an influential study~\cite{MR1555256} on homeomorphisms
of hyperbolic surfaces in which fixed points
are classified in terms of behaviour on the
universal covering space.
In the case of a compact manifold of dimension at least~$3$ (or a
polyhedron satisfying some additional natural
hypotheses), this lower bound is the
best one possible in
that~$N(f)=\min\{\fix_g(1)\mid g{\sim}f\}$
by work of Jiang~\cite{MR685755}.
We refer to that monograph for an extensive
treatment of Nielsen fixed point theory,
and to a survey by Jiang and Zhao~\cite{MR3735835}
and a historical survey
by Brown~\cite{MR1674916} for thorough
treatments.}

In general the sequence~$(N(f^n))$ is not a Dold sequence. However, this does happen for all maps  on a given space if certain topological restrictions on the space are imposed.
{We mention here that there is a minor error
in~\cite[Ex.~11.3]{MR3303933},
where it is claimed that for an orientation-reversing
homeomorphism of~$\S^1$ the Dold congruence
fails (the term corresponding to the
even factor~$10=2\cdot5$ was omitted from a sum over the divisors
of~$90$).
In fact for a continuous map~$f\colon\S^1\to\S^1$ of the
circle, it is known
that
\[
N(f^n)=|\lefschetz({f^n})|=|1-d^n|,
\]
where~$d$ is the degree of the map.} Clearly~$a_n=d^n$ for all~$n\ge1$
defines a Dold sequence~$(a_n)$,
as it is
the sequence of traces of the~$1\times 1$ matrices~$[d]^n$.
Consider now the sequence
with~$n$th term
given by~$N(f^n)=|1-d^n|$.
If~$d\geq 1$, then~$|1-d^n|= d^n-1$ is a Dold sequence.
If~$d<0$ then~$d=-b$, where~$b>0$, and
\[
|1-d^n|=|1-(-b)^n|=b^n-(-1)^n=b^n +\eper_1(n)-\eper_2(n)
\]
is a Dold sequence
as it is a sum of Dold sequences.

On the other hand,
let us observe that~$(|L(f^n)|)$ is not a
Dold sequence in general.
For example,~$(a_n)=(-\eper_2+\eper_3)$ is
a Lefschetz sequence by Lemma~\ref{perDold-Lef},
but~$(|a_n|)=(0,2,3,2,0,1,\ldots)$ does not satisfy
the Dold congruence modulo~$6$.

The simplest examples of maps
whose Nielsen numbers do
not satisfy the Dold congruences
are found in the class of maps of simply-connected spaces.

\begin{example}\label{ex-simply}
Let~$f$ be a map of a simply-connected compact space~$X$.
Then
\[
N(f)
=
\begin{cases}
0&\mbox{if }L(f)=0;\\
1&\mbox{if }L(f)\neq0.
\end{cases}
\]
Thus, to find an example of map~$f$
for which~$(N(f^n))_n$ is not
a Dold sequence, it is enough to find a map~$f$ such that
\begin{equation}\label{nonDold}
\left.
\begin{aligned}
\lefschetz(f)&=0\\
\lefschetz({f^2})&\neq 0
\end{aligned}
\right\}
\end{equation}
and for this we may take a
homeomorphism~$f$ of the~$2$-sphere~$S^2$ that changes the orientation.
Then~$\lefschetz({f^n})=1+(-1)^n$,
and the Lefschetz numbers satisfy~\eqref{nonDold}.
\end{example}

On the other hand, for many other classes of spaces
{beyond the circle}
the sequence~$(N(f^n))$ is a
Dold sequence for all continuous maps.

\begin{proposition}\label{ex-Klein}
If~$f$ is a map of a Klein bottle~$K$,
then~$(N(f^n))$ is a Dold sequence.\index{Klein bottle}
\end{proposition}

\begin{proof}
The Nielsen numbers of iterations in the case of Klein bottle may be
expressed in terms of generators of the fundamental
group (see the works of
Kim, Kim, and Zhao~\cite{MR2410253}
or Llibre~\cite{MR1197046} for the details). A
consequence of these
calculations is that there exist integers~$u$ and~$v$ such that
\[
N(f^n)
=
\begin{cases}
|u^n\cdot(v^n-1)|&\mbox{if }|u|>1;\\
|v^n-1|&\mbox{if }|u|\le1.
\end{cases}
\]
As we mentioned above,
the sequences~$(a_n)$ and~$(b_n)$,
defined by
the relations~$a_n=|u|^n$ and~$b_n=|1-v^n|$
for all~$n\ge1$ are Dold sequences. This shows
that~$(N(f^n)$ is a Dold sequence since
the property is closed under products (as discussed in
the remark at the end of
Section~\ref{sectionRelatingDoldAndRealizableSequences}).
\end{proof}

Although~$(|L(f^n)|)$ is not always a Dold sequence,
nevertheless it belongs to that class  if
\begin{equation}\label{modulus-Lef}
L(f^n)=\det(I-A^n),
\end{equation}
for some~$k\times k$ integer-valued matrix~$A$.
This follows from the fact that it is realised
by a self-map of a torus. Indeed,
taking~$f\colon\T^k\to\T^k$, a toral map induced by
the linear map~$A$,
we know that~$f^n$ has exactly~$|\det(I-A^n)|$ fixed points
for each~$n\ge1$ (see, for example,~ \cite{MR2189944}).
The spaces for which~$N(f^n)$
and~$|L(f^n)|$ agree and Lefschetz numbers satisfy~\eqref{modulus-Lef} include, among others,  self-maps of nilmanifolds and some solvmanifolds.
Thus, in these cases,
Nielsen numbers of iterations
are Dold sequences. Similar arguments may be applied to
so-called infra-solvmanifold of type~$R$
by work of Fel{\cprime}shtyn and Lee~\cite[Th.~11.4]{MR3303933}
(moreover, in this setting ~$N(f^n)=R(f^n)$ for all~$n\ge1$,
showing that~$(R(f^n))$ is also a Dold sequence).
Various topological consequences of these facts are discussed by
Fel{\cprime}shtyn and Troitsky~\cite{MR2465448} and by
Fel{\cprime}shtyn and Lee~\cite{MR3546664,MR3777481}.

A similar theorem holds for solv- and infra-nilmanifolds under
additional assumptions.

\begin{theorem}[Kwasik {\&}
Lee~\protect{\cite[Th.~2]{MR972137},~\cite{MR1123659}}]
Let~$f$ be a continuous map
of a solvmanifold or infra-nilmanifold,
and assume that~$f$ is homotopically periodic
(that is,~$f^k$ is homotopic to the
identity for some~$k>1$).
Then~$(N(f^n))$ is a Dold sequence.
\end{theorem}

This follows from the fact that~$L(f^n)=N(f^n)$ for all~$n\ge1$ for
any map of this type.

\begin{remark}
The Reidemeister number may be also defined in a purely
algebraic way for a group endomorphism~$\phi$
(as the number of~$\phi$-conjugacy clases).
In many cases the sequence of
Reidemeister number of~$\phi^n$ is a Dold sequence
(for details,
we refer
to the works of
Fel{\cprime}shtyn and
Troitsky~\cite{MR1697460, MR2377135, MR2465448} and the
references therein).
\end{remark}



\end{document}